\documentclass[12pt]{amsart}

\usepackage{amssymb,verbatim,graphicx,bm}

\usepackage[mathscr]{eucal}
\usepackage{enumerate}
\usepackage{tikz-cd}
\usepackage{cite}
\usepackage{hyperref}
\makeatletter
\def\mathcolor#1#{\@mathcolor{#1}}
\def\@mathcolor#1#2#3{%
  \protect\leavevmode
  \begingroup
    \color#1{#2}#3%
  \endgroup
}
\makeatother

\newtheorem{defn}{Definition}
\newtheorem{lem}{Lemma}
\newtheorem{prop}{Proposition}
\newtheorem{thm}{Theorem}
\newtheorem{cor}{Corollary}
\newtheorem{cor-proof}{Corollary of the Proof}
\newtheorem{remark}{Remark}
\newtheorem{ex}{Example}

\newtheorem{proc contact}{Procedure Contact}

\newcommand{\field}[1]{\mathbb{#1}}

\newcommand{\partiald}[2]{\displaystyle\frac{\partial#1}{\partial#2}}
\newcommand{\partialds}[3]{\displaystyle\frac{\partial^2#1}{\partial#2\partial#3}}
\newcommand{\partialdss}[2]{\displaystyle\frac{\partial^2#1}{\partial #2^2}}

\newcommand{\mcal}[1]{\mathcal{#1}}
\newcommand{\wt}[1]{\widetilde{#1}}

\newcommand{\pr}[2]{{#1}^{(#2)}}

\newcommand{\be}{\textbf{e}}

\begin{document}

\title{On Some Local Geometry of Bi-Contact Structures}

\author{Taylor J. Klotz and George R. Wilkens}
\address{Department of Applied Mathematics, Universty of Colorado at Boulder, 1111 Engineering Center, ECOT 225
526 UCB, Boulder, CO 80309-0526}
\email{taylor.klotz@colorado.edu,grw@math.hawaii.edu}


\begin{abstract}
We investigate the local geometry of a pair of independent contact structures on 3-manifolds under maps that independently preserve each contact structure. We discover that such maps are homotheties on the contact 1-forms and we discover differential invariants associated to such structures under these equivalences. This allows us to generalize the notion of contact circles and (equilateral) hyperbolas to contact ellipses and hyperbolas. Moreover, these invariants may sometimes be used to define a complete local normal form and in at least one case are related to symplectic structures through a natural $e$-structure on a bundle arising from the Cartan equivalence method. Finally, there is a type of natural Riemannian metric (but not necessarily the well-known associated contact metric) and we discover certain curvatures may be written in terms of the bi-contact differential invariants.
\end{abstract} 

\maketitle
\tableofcontents
\section{Introduction}\label{sec-intro}
Contact geometry and topology has become a large industry of mathematical research in recent decades, especially in the context of additional compatible structures \cite{RiemannianContactSurvey}. Most results that do not involve additional structures are topological in nature due to the classical theorem of Pfaff, which says that there exists a local coordinate system $(x^i,u,p_i)$ about any point $p$ in a contact manifold $M$ such that the contact 1-form $\theta$ has the expression $\theta=du-p_i\,dx^i$. In other words, all (co-orientable) contact 1-forms have the same local geometry and geometric questions become global ones concerning a contact structure's associated Legendrian submanifolds, or other interesting embedding properties such as over-twisted disks.

However, if one were to simultaneously consider two linearly independent contact structures on the same manifold $M$, known as \textit{bi-contact structures}\cite{Mitsumatsu}, then the Pfaff/Darboux normal form doesn't necessarily apply to both contact structures simultaneously. It is then worth asking if such a normal form exists or if instead local invariants arise out of equivalences determined by maps which simultaneously preserve both contact structures as contactomorphisms. 

\begin{defn}
Let $M$ be an orientable 3-manifold with two contact distributions $\mcal{D}_1$ and $\mcal{D}_2$ such that $\mcal{D}_1(p)\neq \mcal{D}_2(p)$ for every $p\in M$. Then we call $(\mcal{D}_1,\mcal{D}_2)$ a \textbf{bi-contact structure}. Whenever, for $i=1,2$, $\mcal{D}_i=\ker\omega^i$ for 1-forms $\omega^i$, then we may refer to $(\omega^1,\omega^2)$ as the bi-contact structure. 
\end{defn}

Bi-contact structures are not to be confused with so-called \textit{contact pairs}, which were introduced with such terminology in \cite{BandeContactCouplesAMS}, but were once called bi-contact structures in \cite{BlairDiffBiContactNotion}. However, we will discover that there is some relation between certain bi-contact structures on 3-manifolds and symplectic couples/pairs on an associated 4-manifold, see \ref{symp-prop}. The notion of symplectic couple was introduced in the work of Geiges in \cite{GeigesSymplecticCouple}. 

We also want to draw attention to the work in \cite{Jackman-Bi-Contact}, which studies essentially the same problem as in this manuscript but with a stronger lean towards local normal forms and symmetries and cites an earlier version of this very paper, which first appeared on the arxiv in December 2023. Of the primary differences between \cite{Jackman-Bi-Contact} and this work is the classification of different bi-contact structures in light of the existing literature. Moreover, we investigate the Riemannian geometry of some invariant metrics in special cases of bi-contact structures. 

Many of our results will be local, and when they are not local, we will assume our bi-contact structures are both co-orientable. As such, it may be assumed that all bi-contact structures mentioned here are co-orientable. In fact, we will study bi-contact structures via differential forms generally. 

\begin{defn}
A map $\Phi:M\to M$ is called a \textbf{bi-contactomorphism} if $\Phi$ if a contactomorphism for both $\mcal{D}_1$ and $\mcal{D}_2$ simultaneously. 
\end{defn}

Bi-contact structures have been studied since at least the early 90's. Indeed, as far as the authors have found, it was Mitsumatsu who introduced the concept in \cite{Mitsumatsu}, but with an additional requirement that the two contact forms induce opposite orientations. We dispense with this requirement as we are interested in studying the equivalence problem with respect to the above mentioned bi-contactomorphisms. It turns out that the opposing orientations will actually define a particular subclass of bi-contact structures that is detected through the first step of torsion normalization in Cartan's method of equivalence. 

Among the bi-contact structures, we will introduce three general cases in Definition \ref{elliptic-hyperbolic} in Section \ref{C-section}: elliptic bi-contact structures, hyperbolic bi-contact structures, and linear bi-contact structures. The linear case will not be explored in this work but merits additional study. Moreover, as special cases, we will naturally discover bi-contact circles and bi-contact equilateral hyperbolas. Bi-contact circles were first introduced by Geiges and Gonzalo in \cite{GeigesTautCirc}. Perrone studied bi-contact circles further by adding an additional compatible Riemannian metric structure known as an \textit{associated contact metric} \cite{PerroneBiContactCirc}. More recently, Perrone also introduced the notion of bi-contact equilateral hyperbolas in \cite{PerroneBiContactHyp} (which Perrone simply called contact hyperbolas) and again studied such structures with associated contact metrics. Our work does not invoke any special notions of compatibility with other structures; however, we do recover an invariant Riemannian metric. 

Moreover, we detect the special case of $(-\epsilon )$-Cartan structures. Which are a further specialization among bi-contact circles and bi-contact hyperbolas.
\subsection{Contact Ellipses and Hyperbolas}
\begin{defn}
A bi-contact structure is called a \textbf{contact ellipse} if 
\[
\omega_a=a_1\omega^1+a_2\omega^2,
\]
is a contact form for every $a=(a_1,a_2)$ belonging to an ellipse $E_{S^1}$ in $\mathbb{R}^2$. Similarly, we say that a bi-contact structure is a \textbf{contact hyperbola} if 
\[
\omega_a=a_1\omega^1+a_2\omega^2,
\]
is a contact form for every $a=(a_1,a_2)$ belonging to a hyperbola $H^1_{r}$ in $\mathbb{R}^2$ where $r=\pm1$ determines the branch of the hyperbola. In addition, if a contact ellipse $\omega_a$ generates the same volume form $\omega_a\wedge d\omega_a$ independent of $a\in E_{S^1}$ then $\omega_a$ is called a \textbf{taut contact ellipse}. If a contact hyperbola $\omega_a$ generates the same volume form up to sign $\omega_a\wedge d\omega_a=r\omega^1\wedge d\omega^1=-r\omega^2\wedge d\omega^2$ independent of $a\in H^1_{r}$ then $\omega_a$ is called a \textbf{taut contact hyperbola}.
\end{defn}

The special case of taut contact circles and taut contact hyperbolas (for the equilateral hyperbolas) have been studied before. Indeed, 3-manifolds admitting taut contact circles have been completely classified in the compact case \cite{GeigesTautCirc}. Taut contact equilateral hyperbolas were first introduced and partially classified by Perrone in \cite{PerroneBiContactHyp}, but such structures are still not completely understood. Moreover, there is a beautiful theory of contact circles up to conformal/homothety class again by Geiges and Gonzalo in \cite{GeigesModuli},\cite{GeigesTransversely}. As such, bi-contactomorphism has in this way been studied before, but in a way specific to taut contact circles. 
\begin{defn}
A bi-contact structure $(\omega^1,\omega^2)$ is called a \textbf{$(-\epsilon )-$Cartan Structure} if
\[
\omega^1\wedge d\omega^2=\omega^2\wedge d\omega^1=0.
\] 
\end{defn}
The $(-\epsilon )-$Cartan Structures are significantly more rigid than taut contact circles and hyperbolas. 

Since we study the general problem of equivalence under bi-contactomorphisms, we will naturally discover both previous results and new geometry that is independent of additional structure, allowing the study of local bi-contact structures to stand on its own. Moreover, there are natural generalizations, not only in the way of considering contact spheres and hyperboloids in the sense of Zessin \cite{Zessin} and Perrone \cite{PerroneBiContactCirc,PerroneBiContactHyp}, but also in the sense of curves of higher degree polynomials on odd dimensional manifold of dimension 5 and higher, which is a planned subject of future work by the first author. 

Lastly, we mention that there is some work to study the global topology of Anosov flows on 3-manifolds, which are intimately related to bi-contact structures \cite{SalmoiraghiGoodman,HozooriContact, Mitsumatsu,Massoni}, including a development of Floer theory applied to such bi-contact structures \cite{FloerBicontact}. Although we are restricting ourselves to a geometry of bi-contact equivalence, we hope that perhaps some of the structure equations found here are helpful for such efforts. Indeed, in a forthcoming work, we will study applications of these structure equations to Anosov flows and Beltrami fields.

The paper is organized as follows: in Section \ref{sec-eq-prob} we setup the equivalence problem and prove that bi-contactomorphisms are in fact homotheties of the contact 1-forms. Moreover, we discover a local differential invariant $C$ and classify bi-contact structures as elliptic, hyperbolic, or linear based on a quadratic form defined by $C$ and an orientation condition. This quadratic form appears naturally from the equivalence problem and was found by Geiges and Gonzalo in \cite{GeigesTautCirc} and \cite{GeigesCirc}, but only closely studied when it represented a circle of contact structures. We then apply the equivalence method to three cases based on $dC$. In each case, we can either find an $e$-structure locally, with complicated torsion, or we can prolong to an involutive system. In Secton \ref{sec-norm-form} we provide a local normal form in two special cases. In Section \ref{sec-inv-metric} we explore some curvature calculations of natural Riemannian metrics from our structure equations. We discover that some submanifolds have mean curvature in terms of differential invariants of bi-contact structures.

\subsection{Acknowledgements}
The authors have no funding sources to acknowledge. 

\section{The Equivalence Problem}\label{sec-eq-prob}

Here we implement E. Cartan's method of equivalence to study the problem of equivalence of bi-contact structures under bi-contactomorphisms. The uninitiated may consult \cite{CartanBeginners,GardnerEquivalence,OlverEquivalence} for an introduction to this technique.
\subsection{Bi-contact Equivalence}\label{equivalence prob}
 Let $M$ be a three dimensional manifold with coframing $\{\omega^1,\omega^2,\omega^3\}$ such that $\omega^1$ and $\omega^2$ are contact forms that are linearly independent and $\omega^3$ any 1-form such that $\omega^1\wedge \omega^2\wedge \omega^3\neq0$. Thus we should expect that a change of coframe under a bi-contactomorphism should be as follows
\begin{gather*}
 \begin{bmatrix}
 \tilde{\omega}^1\\ \tilde{\omega}^2\\ \tilde{\omega}^3\\ 
 \end{bmatrix}
 =
  \begin{bmatrix}
   a_1 & 0 & 0\\
   0 & a_2 & 0\\
   b_1 & b_2 & b_3\\
  \end{bmatrix}^{-1}
    \begin{bmatrix}
 \omega^1\\ \omega^2\\ \omega^3
 \end{bmatrix}.
\end{gather*}
Let $G_0<GL(3,\field{R})$ be the group given by invertible matrices of the above form. Then the structure equations on the principal $G_0$ bundle $\mcal{B}_0\subset \mcal{F}_{GL(3,\field{R})}$ are
\begin{equation}\label{0-struct-eqs}
 d\begin{bmatrix}
 \omega^1\\ \omega^2\\ \omega^3
 \end{bmatrix}
 =
  -\begin{bmatrix}
   \alpha_1 & 0 & 0\\
   0 & \alpha_2 & 0\\
   \beta_1 & \beta_2 & \beta_3\\
  \end{bmatrix}\wedge
 \begin{bmatrix}
 \omega^1\\ \omega^2\\ \omega^3
 \end{bmatrix}+
\begin{bmatrix}
 T^1_{23}\,\omega^2\wedge\omega^3\\  T^2_{13}\,\omega^1\wedge\omega^3\\ 0
\end{bmatrix}.
\end{equation}
Notice that from these structure equations 
\begin{equation}\label{Contact-condition-adapted-0}
\omega^1\wedge d\omega^1=T^1_{23}\omega^1\wedge \omega^2 \wedge\omega^3\text{ and }\omega^2\wedge d\omega^2=-T^2_{13}\omega^1\wedge \omega^2 \wedge\omega^3.
\end{equation}
The property that $\omega^1$ is adapted to a contact distribution on $M$ is equivalent to the condition that $T^1_{23}$ is never zero on $\mcal{B}_0$. Similarly, $T^2_{13}$ is never zero on $\mcal{B}_0$. We can find a normalization of this torsion by considering the induced action of $G_0$ on the intrinsic torsion. 
Indeed, let us rewrite the structure equations \eqref{0-struct-eqs} as $d\omega=-\theta\wedge\omega+T(\omega\wedge\omega)$ where $\theta$ is the (pseudo)connection form. Now we can check what happens to the intrinsic torsion under the change of coframe defined by $G_0$. Indeed we get that $d(g\,\tilde{\omega})=-\theta g\wedge\tilde{\omega}+T(g\tilde{\omega}\wedge g\tilde{\omega})$ which may be rewritten again as $d\tilde{\omega}=-\tilde{\theta}\wedge \tilde{\omega}+g^{-1}\,T(g\tilde{\omega}\wedge g\tilde{\omega})$ where $\tilde{\theta}=g^{-1}dg+g^{-1}\theta\,g+t(\tilde{\omega})$ where $t(\tilde{\omega})$ is the absorbable torsion of $\tilde{T}$. However, $d\tilde{\omega}=-\tilde{\theta}\wedge\tilde{\omega}+\tilde{T}(\tilde{\omega}\wedge\tilde{\omega})$ so upon comparing intrinsic torsion terms one finds that $\tilde{T}^1_{23}=\frac{a_1}{a_2b_3}\,T^1_{23}$ and $\tilde{T}^2_{13}=\frac{a_2}{a_1b_3}T^2_{13}$. Hence $\tilde{T}^1_{23}\tilde{T}^2_{13}=\frac{1}{b_3^2}T^1_{23}T^2_{13}$ meaning that the sign of $T^1_{23}T^2_{13}$ is fixed. As such, we may normalize the torsion in the following way 
\begin{equation*}
T^1_{23}=1\text{ and }T^2_{13}=\epsilon,
\end{equation*}
which implies that 
\[
T^1_{23}T^2_{13}=\epsilon,
\]
where $\epsilon=\pm1$. This is achieved by taking a reduction of the group $G_0$ so that $a_2=\delta a_1$ and $b_3=\delta$ where $\delta=\pm1$. We now find that the 1-forms $\omega^1$ and $\omega^2$ have the property that
\begin{equation}
\omega^1\wedge d\omega^1=-\epsilon \omega^2\wedge d\omega^2=\omega^1\wedge \omega^2\wedge \omega^3,
\end{equation}
via equations \ref{Contact-condition-adapted-0} and the torsion normalization. 
Let us label $a=a_1$ so that $a_2=\delta  a$. Our reduction of $G_0$ yields a restriction to a subbundle $\mcal{B}_1\hookrightarrow\mcal{B}_0$ with group $G_{1}< G_0$ described by the change of coframe
\begin{equation}\label{1-adapted}
 \begin{bmatrix}
 \tilde{\omega}^1\\ \tilde{\omega}^2\\ \tilde{\omega}^3\\ 
 \end{bmatrix}
 =
  \begin{bmatrix}
   a & 0 & 0\\
   0 & \delta  a & 0\\
   b_1 & b_2 & \delta \\
  \end{bmatrix}^{-1}
    \begin{bmatrix}
 \omega^1\\ \omega^2\\ \omega^3
 \end{bmatrix}.
\end{equation}
Any coframe that transforms according to \eqref{1-adapted} will be called 1-adapted. Notice also that $\mcal{B}_1$ is a 4-fold cover of $M$ from the presence of the sign ambiguities in \eqref{1-adapted}. Specifically, there are four connected components of $\mcal{B}_1$ which arise from combinations of the signs of $a$ and $\delta$. The 1-adapted structure equations are given by
\begin{gather*}
 d\begin{bmatrix}
 \omega^1\\ \omega^2\\ \omega^3
 \end{bmatrix}
 =
  -\begin{bmatrix}
   \alpha & 0 & 0\\
   0 & \alpha & 0\\
   \beta_1 & \beta_2 & 0\\
  \end{bmatrix}\wedge
 \begin{bmatrix}
 \omega^1\\ \omega^2\\ \omega^3
 \end{bmatrix}+
\begin{bmatrix}
T^1_{13}\omega^1\wedge\omega^3+ \omega^2\wedge\omega^3\\  T^2_{23}\omega^2\wedge\omega^3+\epsilon \omega^1\wedge\omega^3\\ 0
\end{bmatrix}.
\end{gather*}
We now have new torsion and notice that $\alpha\mapsto\alpha-\frac{1}{2}(T^1_{13}+T^2_{23})\omega^3$ gives for the torsion components $\frac{1}{2}(T^1_{13}-T^2_{23})\omega^1\wedge\omega^3$ and $-\frac{1}{2}(T^1_{13}-T^2_{23})\omega^2\wedge\omega^3$. Calling $C=-\frac{1}{2}(T^1_{13}-T^2_{23})$ we rewrite the structure equations as
\begin{equation}\label{Struct Eq 1}
 d\begin{bmatrix}
 \omega^1\\ \omega^2\\ \omega^3
 \end{bmatrix}
 =
  -\begin{bmatrix}
   \alpha & 0 & 0\\
   0 & \alpha & 0\\
   \beta_1 & \beta_2 & 0\\
  \end{bmatrix}\wedge
 \begin{bmatrix}
 \omega^1\\ \omega^2\\ \omega^3
 \end{bmatrix}+
\begin{bmatrix}
-C\omega^1\wedge\omega^3+ \omega^2\wedge\omega^3\\  \epsilon \omega^1\wedge\omega^3+ C\omega^2\wedge\omega^3\\ 0
\end{bmatrix}.
\end{equation}
We discover that bi-contactomorphisms are \textit{homotheties} of the bi-contact structure. These homotheties have been mentioned before in \cite{GeigesTautCirc} and \cite{PerroneBiContactHyp}, but only in the context of the torsion function $C=0$.  

Notice that \eqref{Struct Eq 1} have not arisen from a reduction of the structure group $G_1$. Rather, we eliminated the ambiguity in the torsion via torsion absorption. The resulting intrinsic torsion is fundamental to bi-contact structures. 
\begin{prop}
$C$ is an invariant function on $M$.  Moreover, $C$ is the local obstruction for the bi-contact structure to be a taut contact circle or taut contact equilateral hyperbola. 
\end{prop}
\begin{proof}
Notice that 
\begin{equation*}
\begin{aligned}
(d^2\omega^1)\wedge\omega^2+(d^2\omega^2)\wedge\omega^1=-2\,dC\wedge\omega^1\wedge\omega^2\wedge\omega^3,
\end{aligned}
\end{equation*}
and therefore there are functions $C_i$, $i=1,2,3$ on $M$ such that $dC=C_i\omega^i$ and so $C$ is an invariant function on $M$. 
It is then immediate what role the invariant function $C$ plays. Indeed, 
\begin{equation}
\omega^1\wedge d\omega^2+\omega^2\wedge d\omega^1=2C\,\omega^1\wedge\omega^2\wedge\omega^3,
\end{equation}
and so by easy theorems in \cite{GeigesTautCirc,GeigesCirc} and \cite{PerroneBiContactCirc,PerroneBiContactHyp} we have that the invariant function $C$ is an obstruction to being a contact circle or equilateral hyperbola.
\end{proof}

That $\omega^1\wedge d\omega^2+\omega^2\wedge d\omega^1$ vanishing defines contact circles and equilateral hyperbolas is a part of the definition of such structures. What is interesting is that this arises naturally when considering equivalence under bi-contactomorphisms.  
%%%%%%%%%%%%%%%%%%%%%%%%%%%%%%%%%%%%%%%%%%%%%%%%%%%%%%%%%%%%%%%%%%%%%%%%%%%%%%%
\subsection{The invariant function \texorpdfstring{$C$}{C}}\label{C-section}
We investigate some of the implications of the invariant function $C$ and its exterior derivative $dC$. It naturally leads us to introduce the more general notion of elliptic bi-contact structures and hyperbolic bi-contact structures. Let $(\omega^1,\omega^2)$ be a bi-contact structure on a 3-manifold $M$. Now let $\omega_a=a_1\omega^1+a_2\omega^2$ for some real numbers $a_1$ and $a_2$. In order for $\omega_a$ to be a contact form, we need it to induce a volume form on $M$. From the structure equations \eqref{Struct Eq 1} we have  
\[
\omega_a\wedge d\omega_a=(a_1^2-\epsilon  a_2^2+2Ca_1a_2)\Omega,
\]
where $\Omega=\omega^1\wedge\omega^2\wedge\omega^3=\omega^1\wedge d\omega^1$. Therefore, in order that $\omega_a$ be a contact form, we need simply avoid the zero locus of the quadratic form $\mathcal{P}_C(a):=a_1^2-\epsilon a_2^2+2Ca_1a_2$. 
\begin{defn}\label{elliptic-hyperbolic}
Given a bi-contact structure $(\omega^1,\omega^2)$ on a 3-manifold $M$, if the $(a_1,a_2)$ for $\omega_a$ are chosen in such a way that $\mathcal{P}_C(a)=r$ for $r=\pm1$ then we say the bi-contact structure is \textbf{taut}. Moreover, if $\mathcal{P}_C(a)$ is a definite quadratic form then we call the bi-contact structure \textbf{elliptic}. If $\mathcal{P}_C(a)$ is indefinite then we say the bi-contact structure is \textbf{hyperbolic}. Finally, if $\mathcal{P}_C(a)$ is degenerate then we say that the bi-contact structure is \textbf{linear}. 
\end{defn}
Thus the following proposition is immediate. 
\begin{prop}\label{sig-to-taut}
Let $(\omega^1,\omega^2)$ be a bi-contact structure on 3-manifold $M$. 
\begin{enumerate}
\item If $\epsilon =-1$ and $|C|<1$ then the bi-contact structure is elliptic and admits a taut contact ellipse.
\item If $\epsilon =-1$ and $|C|=1$ then the bi-contact structure is linear and admits a taut contact line. 
\item If $\epsilon =1$ or $|C|>1$ with $\epsilon =-1$ then the bi-contact structure is hyperbolic and admits a taut contact hyperbola $\omega_a$. Additionally, if the taut contact hyperbola is induced by the case $\epsilon =1$ then $\omega^1\wedge d\omega^1=-\omega^2\wedge d\omega^2$ and if the taut contact hyperbola is induced by the case $|C|>1$ and $\epsilon =-1$ then $\omega^1\wedge d\omega^1=\omega^2\wedge d\omega^2$.  
\end{enumerate}
\end{prop}

Moreover, it is interesting to note that we need not restrict ourselves entirely to the case that $(a_1,a_2)$ are constant on $M$. 
\begin{prop}\label{var-quadratic-form}
Let $\omega^3$ be a 1-form completing a bi-contact structure $(\omega^1,\omega^2)$ on a 3-manifold $M$ to a local co-framing. Then contact circles and contact hyperbolas may be parameterized by two functions $a_1,a_2 \in C^2(M)$ such that 
\[
a_1a_{2,3}-a_2a_{1,3}=0
\] 
where $a_{i,3}$ are the covariant derivatives of $a_i$ in the direction $e_3$ which is dual to $\omega^3$. Moreover, if $\omega^3$ locally defines a foliation of $M$ then the ratio $a_1/a_2$ is a first integral of $e_3$. 
\end{prop}
\begin{proof}
Let $\omega_a=a_1\omega^1+a_2\omega^2$. Then 
\[
\omega_a\wedge d\omega_a = \mcal{P}_C(a)\Omega+\omega_a\wedge (da_1\wedge \omega^1+da_2\wedge \omega^2), 
\]
where $\Omega = \omega^1\wedge \omega^2 \wedge \omega^3$. Expanding the $da_i$ terms and computing the wedge product immediately yields the claimed differential equation for $a_1$ and $a_2$. Note also that 
\[
\left(\frac{a_1}{a_2}\right)_3 = \frac{a_1a_{2,3}-a_2a_{1,3}}{a_2^2},
\]
and so when $\omega^3$ defines a foliation we have that $a_1/a_2$ is constant along integral curves of $e_3$. 
\end{proof}
The invariant function $C$ is the first obstruction of a \textit{given} 1-adapted (i.e. the volume forms are fixed) bi-contact structure to admit taut contact circles or taut equilateral hyperbolas. However, as Theorem \ref{unit-taut} below demonstrates, the derivative of $C$ is the true obstruction to the existence of taut contact circles and taut contact equilateral hyperbolas outside of the homothety/bi-contactomorphism class of a given bi-contact structure. 
\begin{prop}\label{unit-taut}
Let $(\omega^1,\omega^2)$ be a bi-contact structure on 3-manifold $M$ such that $dC\wedge \omega^1\wedge \omega^2 =0$ on $M$. If the bi-contact structure is elliptic then $M$ admits a taut contact circle that is not bi-contactomorphic to $(\omega^1,\omega^2)$. If the bi-contact structure is hyperbolic then $M$ admits a taut contact equilateral hyperbola that is not bi-contactomorphic to $(\omega^1,\omega^2)$. 
\end{prop}
\begin{proof}
We construct such taut bi-contact structures directly in terms of the given bi-contact structure as though we were diagonalizing the quadratic form $\mathcal{P}_C(a)$. Let us first consider the case of $\epsilon =-1$. Indeed, let
\begin{equation}\label{taut-trans}
\begin{aligned}
\eta^1&=\frac{1}{\sqrt{|1+C|}}\frac{(\omega^1+\omega^2)}{\sqrt{2}},\\
\eta^2&=\frac{1}{\sqrt{|1-C|}}\frac{(\omega^1-\omega^2)}{\sqrt{2}}. 
\end{aligned}
\end{equation}
First note that 
\[
\eta^1\wedge\eta^2=-\frac{1}{\sqrt{|1-C^2|}}\omega^1\wedge\omega^2,
\]
and therefore to have $(\eta^1, \eta^2)$ to be a 1-adapted bi-contact structure, we need to define $\eta^3=-\sqrt{|1-C^2|}\,\,\omega^3$ so that the volume forms from $(\eta^1,\eta^2)$ and $(\omega^1,\omega^2)$ agree. Now we need simply check that the 1-forms $\eta^1$ and $\eta^2$ define taut contact circles for $|C|<1$ and taut contact equilateral hyperbolas when $|C|>1$. Routine calculation yields the following 
\begin{equation*}
\begin{aligned}
d\eta^1&=\frac{1}{\sqrt{|1+C|}}\frac{d\omega^1+d\omega^2}{\sqrt{2}}+\frac{\text{sgn}(1+C)}{2}\frac{dC}{|1+C|^{3/2}}\wedge \frac{(\omega^1+\omega^2)}{\sqrt{2}},\\
d\eta^2&=\frac{1}{\sqrt{|1-C|}}\frac{d\omega^1-d\omega^2}{\sqrt{2}}-\frac{\text{sgn}(1-C)}{2}\frac{dC}{|1-C|^{3/2}}\wedge \frac{(\omega^1-\omega^2)}{\sqrt{2}},\\
\eta^1\wedge d\eta^1&=\frac{1}{2|1+C|}(\omega^1\wedge d\omega^1+\omega^1\wedge d\omega^2+\omega^2 \wedge d\omega^1+\omega^2 \wedge d\omega^2),\\
\,&=\frac{1+C}{|1+C|}\Omega=\text{sgn}(1+C)\Omega,\\
\eta^2\wedge d\eta^2&=\frac{1}{2|1-C|}(\omega^1\wedge d\omega^1-\omega^1\wedge d\omega^2-\omega^2 \wedge d\omega^1+\omega^2 \wedge d\omega^2),\\
\,&=\frac{1-C}{|1-C|}\Omega=\text{sgn}(1-C)\Omega,\\
\eta^1\wedge d\eta^2&=-\frac{1}{2|1-C^2|}\left(\omega^1\wedge d\omega^2-\omega^2\wedge d\omega^1\right)+\frac{1}{4}\frac{\text{sgn}(1-C)dC}{|1-C|\sqrt{|1-C^2|}}\wedge \omega^1\wedge \omega^2,\\
\eta^2\wedge d\eta^1&=\frac{1}{2|1-C^2|}\left(\omega^1\wedge d\omega^2-\omega^2\wedge d\omega^1\right)+\frac{1}{4}\frac{\text{sgn}(1+C)dC}{|1+C|\sqrt{|1-C^2|}}\wedge \omega^1\wedge \omega^2. 
\end{aligned}
\end{equation*}
As such, we find that 
\[
\eta^1\wedge d\eta^2+\eta^2\wedge d\eta^1=C_3\,\frac{\text{sgn}(1-C)|1+C|+\text{sgn}(1+C)|1-C|}{2(|1-C^2|)^{3/2}}\,\Omega
\]
and hence if $(a_1,a_2)\in S^1\subset \mathbb{R}^2$ with $|C|<1$ for $\eta_a=a_1\eta^1+a_2\eta^2$ then 
\begin{equation*}
\begin{aligned}
\eta_a\wedge d\eta_a=\left(1+\frac{a_1a_2\,C_3}{(1-C^2)^{3/2}}\right)\Omega. 
\end{aligned}
\end{equation*}
In case $(a_1,a_2)\in H^1_{r}$ with $|C|>1$ then 
\begin{equation*}
\begin{aligned}
\eta_a\wedge d\eta_a=\left(1-\frac{a_1a_2\,C_3}{(C^2-1)^{3/2}}\right)\Omega. 
\end{aligned}
\end{equation*}
In both cases we see that if $C_3=0$ on $M$, equivalently $dC\wedge\omega^1\wedge\omega^2=0$, then the contact forms $\eta^1$ and $\eta^2$ form a contact circle or a contact equilateral hyperbola. 
To handle the remaining hyperbolic case, we will apply the same argument, but with different 1-forms $\eta^1$ and $\eta^2$. In what follows, it is slightly more convenient to express $C=\sinh(\theta)$ and work with hyperbolic/exponential functions of some $\theta\in C^\infty(M)$. Let 
\begin{equation*}
\begin{aligned}
\eta^1&=\frac{\cosh\left(\frac{\theta}{2}\right)\omega^1+\sinh\left(\frac{\theta}{2}\right)\omega^2}{\cosh(\theta)},\\
\eta^2&=\frac{-\sinh\left(\frac{\theta}{2}\right)\omega^1+\cosh\left(\frac{\theta}{2}\right)\omega^2}{\cosh(\theta)}.
\end{aligned}
\end{equation*}
Then straightforward calculations yield
\begin{equation*}
\begin{aligned}
d\eta^1&=\frac{\cosh\left(\frac{\theta}{2}\right)d\omega^1+\sinh\left(\frac{\theta}{2}\right)d\omega^2}{\cosh(\theta)}-\frac{\left((\cosh(\theta)+2)\sinh\left(\frac{\theta}{2}\right)\omega^1+(\cosh(\theta)-2)\cosh\left(\frac{\theta}{2}\right)\omega^2 \right)\wedge d\theta}{2\cosh^2(\theta)},\\
d\eta^2&=\frac{-\sinh\left(\frac{\theta}{2}\right)d\omega^1+\cosh\left(\frac{\theta}{2}\right)d\omega^2}{\cosh(\theta)}+\frac{\left((\cosh(\theta)-2)\cosh\left(\frac{\theta}{2}\right)\omega^1-(\cosh(\theta)+2)\sinh\left(\frac{\theta}{2}\right)\omega^2 \right)\wedge d\theta}{2\cosh^2(\theta)},
\end{aligned}
\end{equation*}
and therefore
\begin{equation*}
\begin{aligned}
\eta^1\wedge d\eta^1&=\Omega-\frac{\omega^1\wedge\omega^2 \wedge d\theta}{2\cosh^2(\theta)},\\
\eta^2\wedge d\eta^2&=-\Omega-\frac{\omega^1\wedge\omega^2 \wedge d\theta}{2\cosh^2(\theta)},\\
\eta^1\wedge d\eta^2&=\frac{\omega^2\wedge d\omega^1-\omega^1\wedge d\omega^2}{2\cosh(\theta)}-\frac{1}{2}\text{sech}(\theta)\text{tanh}(\theta)\omega^1\wedge \omega^2\wedge d\theta,\\
\eta^2\wedge d\eta^1&=-\frac{\omega^2\wedge d\omega^1-\omega^1\wedge d\omega^2}{2\cosh(\theta)}+\frac{1}{2}\text{sech}(\theta)\text{tanh}(\theta)\omega^1\wedge \omega^2\wedge d\theta.
\end{aligned}
\end{equation*}
Interestingly $\eta^1$ and $\eta^2$ can be used to define a contact equilateral hyperbola, but not generally taut. Since $dC = \cosh(\theta)d\theta$ it is then immediately clear that $\eta^1$ and $\eta^2$ form a taut contact equilateral hyperbola if and only if $dC\wedge \omega^1\wedge \omega^2=0$ on all of $M$.
\end{proof}

\begin{cor-proof}\label{cor-proof}
Let $M$ be a compact orientable 3-manifold and let $(M,\omega^1,\omega^2)$ be a bi-contact structure such that $dC\wedge \omega^1\wedge \omega^2=0$ everywhere on $M$. If $(\omega^1,\omega^2)$ is elliptic then $M$ is diffeomorphic to the quotient of a Lie group G by a discrete subgroup acting by left multiplication where $G=SU(2),\widetilde{SL}(2,\mathbb{R}),$ or $\widetilde{E}(2)$ where $\tilde{\cdot}$ denotes the universal cover. Moreover, any such 3-manifold $M$ realizes a distinct elliptic bi-contactomorphism class for every $C\in C^1(M)$ such that $dC\wedge \omega^1\wedge \omega^2=0$ everywhere on $M$. 
\end{cor-proof}

\begin{proof}
Note that the transformation \eqref{taut-trans} is well defined on all of $M$ and is of full rank for any $C\in C^1(M)$ such that $dC\wedge\omega^1\wedge \omega^2=0$. So we may always induce a taut contact circle on $M$. Then by the main result of \cite{GeigesTautCirc} -i.e. the classification of compact 3-manifolds admitting taut contact circles- in conjunction with $C$ being an invariant under bi-contactomorphism, the Corollary follows. 
\end{proof}
Before returning to the equivalence problem, we provide a local example on some $\mathbb{R}^3$ of a bi-contact structure (both hyperbolic and elliptic depending on the value of $z$ and $\epsilon $) with $dC\wedge \omega^1\wedge\omega^2\neq 0$. This demonstrates that there exist bi-contact structures that are beyond the scope of Corollary \ref{cor-proof}. 
\begin{ex}\label{hyp-ex-C3-nonzero}
Let $M=\mathbb{R}^3$ with coordinates $(x,y,z)$ and let $C=z$ with $C_3=C_3(z)$ arbitrary as a non-zero function of $z$. Note that this means $dC\wedge dC_3=0$. 
\begin{equation}
\begin{aligned}
\omega^1&=dx+dy-\frac{q_1(x,y,z)}{C_3(z)}dz,\\
\omega^2&=dx-dy+\frac{q_2(x,y,z)}{C_3(z)}dz,\\
\omega^3&=\frac{dz}{C_3(z)},
\end{aligned}
\end{equation}
where
\begin{equation}
\begin{aligned}
q_1(x,y,z)&=(x+y)z-(x-y),\\
q_2(x,y,z)&=\epsilon (x+y)+(x-y)z.
\end{aligned}    
\end{equation}
\end{ex}
It is the case that 
\[
\omega^1\wedge d\omega^2+\omega^2\wedge d\omega^1=2z\omega^1\wedge\omega^2\wedge\omega^3.
\]
Matters being so, we find
\[
dC\wedge \omega^1\wedge\omega^2 = C_3(z)\omega^1\wedge\omega^2\wedge\omega^3,
\]
which is nonzero. Curiously, this gives an example of what one might call a \textit{mixed} bi-contact structure in that for $\epsilon=-1$ and $|z|<1$ it is a contact ellipse but when $|z|>1$ it is a contact hyperbola. Such structures warrant additional investigation which we do not pursue here. 
%%%%%%%%%%%%%%%%%%%%%%%%%%%%%%%%%%%%%%%%%%%%%%%%%%%%%%%%%%%%%%%%%%%%%%%%%%%%%%%

\subsection{Non-Constant C} 

Returning to the equivalence problem, it is easy to see that the 1-form $\alpha$ is unique since any other choice $\alpha'$ would give $(\alpha-\alpha')\wedge \omega^1=0$ and $(\alpha-\alpha')\wedge\omega^2=0$. However, Cartan's lemma gives 
\begin{equation}
\begin{bmatrix}\beta_1\\ \beta_2\end{bmatrix}=\begin{bmatrix} \beta'_1\\ \beta'_2 \end{bmatrix}+\begin{bmatrix} h_{11} & h_{12}\\ h_{12} & h_{22} \end{bmatrix}\begin{bmatrix} \omega_1\\ \omega_2 \end{bmatrix}. 
\end{equation}
It is then immediate that the tableaux (call it $A$) is not involutive since $\dim(\pr{A}{1})=3$ but the non-zero Cartan characters are $s_1=2$ and $s_2=1$. 

At this stage, we may either prolong to a larger bundle where the tautological 1-form includes the Lie algebra 1-forms or we may seek additional reduction of the structure group by investigating the action on the 1-jet of the invariant function $C$. The second option turns out to save us some effort for non-constant $C$ and is more in spirit with the work of E. Cartan's 1908 paper \cite{Cartan1908}. Moreover, thanks to Theorem \ref{unit-taut}, we have direct geometric interpretations for the invariant $C_3$ so that understanding the action on $dC$ provides us with additional insight. 

We now check the action of the group $G_1$ on $dC$, where
\begin{equation*}
dC=C_1\tilde{\omega}^1+C_2\tilde{\omega}^2+C_3\tilde{\omega}^3. 
\end{equation*}
We can easily see that under a change of coframe 
\begin{equation*}
dC=\left(\frac{C_1}{a}-\frac{ b_1 C_3}{a}\right)\omega^1+\left(\frac{ C_2}{\delta  a}-\frac{b_2 C_3}{\delta  a}\right)\omega^2+ C_3\omega^3.
\end{equation*}
We now consider the different possibilities of each $C_i$ for $1\leq i\leq 3$ and their associated reductions of the structure group.
\begin{enumerate}
\item \textbf{Case 1: $dC\wedge \omega^1\wedge\omega^2=0$ ($C_3 = 0$)}. In this scenario, we are left with the action of $G_1$ on $dC$ as 
\begin{equation*}
dC=\frac{C_1}{a}\omega^1+\frac{C_2}{\delta  a}\omega^2.
\end{equation*}
In addition to reducing the group parameter $a$, we also discover that we may normalize
\[
C_1^2+C_2^2=1,
\] 
and as such we may take $C_1=\cos(\xi)$ and $C_2=\delta  \sin(\xi)$ for some function $\xi$ on $M$. However, the group parameters $b_1$ and $b_2$ are still undetermined. In fact, the tableaux for the new structure equations is easily seen to be involutive. That said, we have introduced additional torsion into the structure equations. Indeed, 
\[
\alpha = A_1\omega^1+A_2\omega^2+A_3\omega^3,
\]
so that our new structure equations take the form
\begin{equation*}
\begin{aligned}
d\omega^1&=A_2\omega^1\wedge \omega^2+(A_3-C)\omega^1\wedge\omega^3+\omega^2\wedge\omega^3,\\
d\omega^2&=-A_1\omega^1\wedge\omega^2+\epsilon  \omega^1\wedge\omega^3+(A_3+C)\omega^2\wedge\omega^3,\\
d\omega^3&=-\beta_1\wedge \omega^1-\beta_2\wedge \omega^2.
\end{aligned}
\end{equation*}
Now we look to reduce again by normalizing the new torsion. The action of the current group on this torsion is 
\begin{equation*}
\begin{aligned}
A_1&\mapsto A_1+b_2-(A_3+C)b_1,\\
A_2&\mapsto A_2+b_1-(A_3-C)b_2,
\end{aligned}
\end{equation*}
Thus, we can translate away the torsion terms $A_1$ and $A_2$ and reduce the group again to achieve an $e$-structure as long as $A_3^2-C^2\neq 1$ (we do not explore the alternative case). Upon this reduction the forms $\beta_1$ and $\beta_2$ become basic so that
\begin{equation*}
\begin{aligned}
\beta_1&=B_5\omega^1+\frac{1}{2}(B_3+B_4)\omega^2+B_2\omega^3,\\
\beta_2&=\frac{1}{2}(B_4-B_3)\omega^1+B_6\omega^2+B_1\omega^3,
\end{aligned}
\end{equation*}
where we have chosen the labeling of the $B$ coefficients for simplicity of presentation in the structure equations below
\begin{equation}
\begin{aligned}
d\omega^1&=(A_3-C)\omega^1\wedge\omega^3+\omega^2\wedge\omega^3,\\
d\omega^2&=\epsilon  \omega^1\wedge\omega^3+(A_3+C)\omega^2\wedge\omega^3,\\
d\omega^3&=B_3 \omega^1\wedge\omega^2+B_2\omega^1\wedge \omega^3+B_1\omega^2 \wedge \omega^3.
\end{aligned}
\end{equation}
We have now achieved an $e$-structure, but with complicated orbits, and equivalences determined by all the remaining torsion. 
\item \textbf{Case 2: $dC\wedge\omega^1\wedge \omega^2\neq0$ and $dC\wedge dC_3\neq0$}. In this scenario, we are left with the action of $G_1$ on $dC$ as 
\begin{equation*}
dC=-\frac{C_1-b_1 C_3}{a}\omega^1-\frac{C_2-b_2 C_3}{\delta a}\omega^2+ C_3\omega^3.
\end{equation*}
We see immediately that we can adapt our coframe so that we have the group reduction $b_1=b_2=0$ by normalizing the $C_1=C_2=0$. We are still left with one group parameter and the 1-forms $\beta_1$ and $\beta_2$ become basic, so that 
\begin{equation*}
\begin{aligned}
\beta_1&=B_5\omega^1+\frac{1}{2}(B_3+B_4)\omega^2+B_2\omega^3,\\
\beta_2&=\frac{1}{2}(B_4-B_3)\omega^1+B_6\omega^2+B_1\omega^3,
\end{aligned}
\end{equation*}
where we have chosen the labeling of the $B$ coefficients for simplicity of presentation in the structure equations below,
\begin{equation*}
\begin{aligned}
d\omega^1&=-\alpha\wedge \omega^1-C\omega^1\wedge\omega^3+\omega^2\wedge\omega^3,\\
d\omega^2&=-\alpha\wedge \omega^2+\epsilon \omega^1\wedge\omega^3+C\omega^2\wedge\omega^3,\\
d\omega^3&=B_3\omega^1\wedge\omega^2+B_2\omega^1\wedge\omega^3+B_1\omega^2\wedge\omega^3.
\end{aligned}
\end{equation*}
As such, we have reduced the group $G_1$ to $G_2$ and we are now restricted to a subbundle $\mcal{B}_2\subset\mcal{B}_1$. The tableaux is still not involutive since $\dim(A^{(1)})=0$ and the first Cartan character is nonzero. Thus, we seek an additional reduction of the structure group. First, however, we may notice that $B_3=0$. Indeed, since $dC=C_3\omega^3$, we can check the relations determined by $d^2C=0$. Doing so, we find that
\begin{equation}\label{d^2C=0}
0=dC_3\wedge \omega^3+C_3d\omega^3,
\end{equation}
meaning that $B_3$ must be zero. Since we are assuming that $dC\wedge dC_3\neq0$ on $M$ then $B_1$ and $B_2$ cannot simultaneously vanish. As such, we now check the action of the group on the remaining torsion.
\[
\begin{bmatrix} B_1\\ B_2 \end{bmatrix} \mapsto \begin{bmatrix} a & 0 \\ 0 & \delta  a  \end{bmatrix}\begin{bmatrix} B_1\\ B_2 \end{bmatrix}.
\]
We recognize that $a^2(B_1^2+B_2^2)$ may be normalized to $1$ and so we may reduce the structure group again so that $a=1$. This means there is a function $\zeta$ on $M$ such that $B_2=\delta \sin(\zeta)$ and $B_1=\cos(\zeta)$.  We also observe the following relations, 
\begin{equation*}
C_{31}=-\delta C_3\sin(\zeta)\text{ and }C_{32}=-C_3\cos(\zeta).
\end{equation*}
The group reduction introduces additional torsion since $\alpha$ is now basic. Indeed, let
\[
\alpha=A_1\omega^1+A_2\omega^2+A_3\omega^3.
\]
Then the structure equations become 
\begin{equation}
\begin{aligned}
d\omega^1&=A_2\omega^1\wedge\omega^2+(A_3-C)\omega^1\wedge\omega^3+\omega^2\wedge\omega^3,\\
d\omega^2&=-A_1\omega^1\wedge\omega^2+\epsilon \omega^1\wedge\omega^3+(A_3+C)\omega^2\wedge\omega^3,\\
d\omega^3&=\delta \sin(\zeta)\omega^1\wedge\omega^3+\cos(\zeta)\omega^2\wedge\omega^3.
\end{aligned}
\end{equation}
We have now reduced to an $e$-structure with complicated torsion. 
\item \textbf{Case 3: $dC\wedge\omega^1\wedge\omega^2\neq0$ and $dC\wedge dC_3=0$ ($C_3\neq 0$ and $B_1=B_2=0$)}. We pick up from Case 2 in the situation where $B_1=B_2=0$ and therefore we no longer have remaining torsion to normalize using the scaling factor $a$. However, since $\alpha$ is uniquely determined, we have an $e$-structure on $\mcal{B}_2$. Let $\alpha=\omega^4$. Then it is immediate from $d^2\omega^1=0,d^2\omega^2=0$ that 
\[
d\omega^4=E\omega^1\wedge \omega^2,
\]
for some $E$ on $\mcal{B}_2$. Moreover, the condition that $d^2\omega^4=0$ yields
\[
dE=E_1\omega^1+E_2\omega^2+2E\omega^4.
\]
Thus, on the bundle $\mcal{B}_2$ we have the following structure equations 
\begin{equation}\label{symp-struct-eqs}
\begin{aligned}
d\omega^1&=\omega^1\wedge \omega^4-C\omega^1\wedge \omega^3+\omega^2\wedge \omega^3,\\
d\omega^2&=\omega^2\wedge \omega^4+C\omega^2\wedge\omega^3+\epsilon \omega^1\wedge\omega^3,\\
d\omega^3&=0,\\
d\omega^4&=E\omega^1\wedge\omega^2.
\end{aligned}
\end{equation}
\end{enumerate}
We summarize the first two cases above as a single Theorem. 
\begin{thm}\label{non-const-C-struct}
Let $(\omega^1,\omega^2,\omega^3)$ be a coframing of a 3-manifold $M$ such that $\omega^1$ and $\omega^2$ form a bicontact structure. Away from the critical set of $C$ we  may adapt the coframing to one of the following structure equations:
\begin{enumerate}
\item If $C_3= 0$ then
\begin{equation}\label{struct-eqs 1}
\begin{aligned}
d\omega^1&=(A_3-C)\omega^1\wedge\omega^3+\omega^2\wedge\omega^3,\\
d\omega^2&=\epsilon  \omega^1\wedge\omega^3+(A_3+C)\omega^2\wedge\omega^3,\\
d\omega^3&=B_3 \omega^1\wedge\omega^2+B_2\omega^1\wedge \omega^3+B_1\omega^2 \wedge \omega^3,\\
dC&=\cos(\xi)\omega^1+\delta \sin(\xi)\omega^2.
\end{aligned}
\end{equation}
\item If $C_3\neq0$ then $C_1= C_2= 0$ and if $dC\wedge dC_3 \neq 0$ then
\begin{equation}\label{struct-eqs 2}
\begin{aligned}
d\omega^1&=A_2\omega^1\wedge\omega^2+(A_3-C)\omega^1\wedge\omega^3+\omega^2\wedge\omega^3,\\
d\omega^2&=-A_1\omega^1\wedge\omega^2+\epsilon \omega^1\wedge\omega^3+(A_3+C)\omega^2\wedge\omega^3,\\
d\omega^3&=\delta \sin(\zeta)\omega^1\wedge\omega^3+\cos(\zeta)\omega^2\wedge\omega^3,\\
dC&=C_3\omega^3.
\end{aligned}
\end{equation}
\end{enumerate}
\end{thm}
\begin{proof}
The structure equations themselves follow from the preceding moving frame calculations. Equations \eqref{struct-eqs 2} are not valid at critical points of $C$ by hypothesis that $C_3\neq 0$. Moreover, the equations \eqref{struct-eqs 1} never allow for a critical point of $C$ by virtue of the fact that $C_1^2+C_2^2=1$. 
\end{proof}
We also include this proposition that extends the structure equations from Theorem \ref{non-const-C-struct}. 
\begin{prop}\label{der-conditions}
Let $(\omega^1,\omega^2,\omega^3)$ be a bi-contact structure as in Theorem \ref{non-const-C-struct}.
\begin{enumerate}
\item If $C_3=0$ and $B_3=0$ everywhere on $U\subset M$ then 
\begin{equation}
\begin{aligned}
d\xi&=-\rho\sin(\xi)\omega^1+\rho\cos(\xi)\omega^2+\xi_3\omega^3,\\
\xi_3&=\frac{1}{2}(\epsilon +1)\cos(2\xi)+C\sin(2\xi)+\frac{1}{2}(1-\epsilon ),\\
A_3&=\frac{1}{2}\partiald{\xi_3}{\xi},\\
\begin{bmatrix} B_1 \\ B_2\end{bmatrix}&=\begin{bmatrix} C-A_3 & 1 \\ -\epsilon  & C+A_3 \end{bmatrix}^{-1}\begin{bmatrix} \cos(\xi) & -\sin(\xi) \\ \sin(\xi) & \cos(\xi) \end{bmatrix}\begin{bmatrix}\displaystyle\frac{\rho}{2}\partiald{A_3}{\xi}-\sin(2\xi) \\ -\displaystyle\frac{1}{2}\partiald{A_3}{C}-\cos(2\xi)\end{bmatrix},\\
\rho&=\frac{\sin(2\xi)}{\xi_3+\epsilon -1}.
\end{aligned}
\end{equation}
\item If $C_3\neq0$ on $U\subset M$ then
\begin{equation}
d\zeta=(W\cos(\zeta)+\delta A_2)\omega^1-(\delta W\sin(\zeta)+\delta A_1)\omega^2+\zeta_3\omega^3
\end{equation}
\end{enumerate}
for some functions $W$ and $\zeta_3$ on $M$.
\end{prop}
\begin{proof}
Let us prove part $(1)$ first. Starting from equations \eqref{struct-eqs 1}, we compute $d^2C$ and let 
$d\xi=\xi_1\omega^1+\xi_2\omega^2+\xi_3\omega^3$. Then
\begin{equation*}
\begin{aligned}
d^2C&=(\xi_1\cos(\xi)+\xi_2\sin(\xi))\,\omega^1\wedge\omega^2 + ((A_3-C)\cos(\xi)-\sin(\xi)(-\epsilon -\xi_3))\,\omega^1\wedge \omega^3\\
\,&+((A_3+C)\sin(\xi)+\cos(\xi)(1-\xi_3))\,\omega^2\wedge \omega^3.
\end{aligned}
\end{equation*}
Thus there is some function $\rho$ on $M$ such that
\begin{equation*}
\begin{aligned}
\xi_1&=-\rho\sin(\xi),\\
\xi_2&=\rho\cos(\xi).
\end{aligned}
\end{equation*}
Moreover, rearranging the remaining 2-form coefficients into a matrix equation yields
\begin{equation*}
\begin{aligned}
\begin{bmatrix} A_3 \\ -\xi_3 \end{bmatrix} & = \begin{bmatrix} \cos(\xi) & -\sin(\xi) \\ \sin(\xi) & \cos(\xi) \end{bmatrix}^{-1}\begin{bmatrix} \cos(\xi) & -\epsilon \sin(\xi)\\ -\sin(\xi) & \cos(\xi) \end{bmatrix}\begin{bmatrix}C \\ 1 \end{bmatrix}.
\end{aligned}
\end{equation*}
As such
\begin{equation*}
\begin{aligned}
\xi_3 & = C \sin(2\xi)+\frac{1}{2}(1+\epsilon )\cos(2\xi)+\frac{1}{2}(1-\epsilon ),\\
A_3 & = C \cos(2\xi)-\frac{1}{2}(1+\epsilon )\sin(2\xi).
\end{aligned}
\end{equation*}
It is easy to observe the slightly simpler expression that is $A_3=\displaystyle{\frac{1}{2}}\partiald{\xi_3}{\xi}$ with $A_3$ thought of as a function of $C$ and $\xi$. We remark as well that
\[
\partiald{A_3}{\xi}=1-\epsilon -2\xi_3
\]
as it is sometimes helpful later on.
Now we need to inspect $d^2\omega^1$ and $d^2\omega^2$. Doing so we find
\begin{equation*}
\begin{aligned}
d^2\omega^1&=((C-A_3)B_1+B_2+\sin(\xi))\Omega-\omega^1\wedge dA_3 \wedge\omega^3,\\
d^2\omega^2&=(-\epsilon  B_1+(C+A_3)B_2+\cos(\xi))\Omega+dA_3\wedge\omega^2\wedge\omega^3,
\end{aligned}
\end{equation*}
so that we need only $(A_3)_1$ and $(A_3)_2$. Of course, 
\[
dA_3 = \frac{1}{2}\left(\partialds{\xi_3}{\xi}{C}dC+\partialdss{\xi_3}{\xi}d\xi\right),
\]
so 
\begin{equation*}
\begin{aligned}
-(A_3)_2 &= \frac{1}{2}\left(-\partialdss{\xi_3}{\xi}\rho \cos(\xi)-\partialds{\xi_3}{\xi}{C}\sin(\xi)\right),\\
(A_3)_1 &= \frac{1}{2}\left(-\partialdss{\xi_3}{\xi}\rho \sin(\xi)+\partialds{\xi_3}{\xi}{C}\cos(\xi)\right).
\end{aligned}
\end{equation*}
The vanishing of $d^2\omega^1$ and $d^2\omega^2$ may now be recast as the equations
\begin{equation*}
\begin{bmatrix} C-A_3 & 1 \\ -\epsilon  & C+A_3 \end{bmatrix}\begin{bmatrix} B_1 \\ B_2 \end{bmatrix}+\begin{bmatrix} \sin(\xi) \\ \cos(\xi) \end{bmatrix}+\frac{1}{2}\begin{bmatrix} \cos(\xi) & -\sin(\xi) \\ \sin(\xi) & \cos(\xi) \end{bmatrix} \begin{bmatrix}-\rho\partialdss{\xi_3}{\xi} \\ \partialds{\xi_3}{\xi}{C} \end{bmatrix}=\begin{bmatrix} 0 \\ 0\end{bmatrix}.
\end{equation*}
Thus 
\begin{equation*}
\begin{bmatrix} B_1 \\ B_2\end{bmatrix}=\begin{bmatrix} C-A_3 & 1 \\ -\epsilon  & C+A_3 \end{bmatrix}^{-1}\begin{bmatrix} \cos(\xi) & -\sin(\xi) \\ \sin(\xi) & \cos(\xi) \end{bmatrix}\begin{bmatrix}\displaystyle\frac{\rho}{2}\partiald{A_3}{\xi}-\sin(2\xi) \\ -\displaystyle\frac{1}{2}\partiald{A_3}{C}-\cos(2\xi)\end{bmatrix}.
\end{equation*}
Next we investigate a $d^2\xi=0$ conditions. Let $d\rho=\rho_1\omega^1+\rho_2\omega^2+\rho_3\omega^3$. Then 
\[
d^2\xi\equiv \left(\rho_1\cos(\xi)+\rho_2\sin(\xi)+\rho^2\right)\omega^1\wedge\omega^2\mod\{\omega^3\}
\]
so we conclude that there exists a function $W$ such that 
\begin{equation*}
\begin{aligned}
\rho_1&=-\rho^2\cos(\xi)-W\sin(\xi),\\
\rho_2&=-\rho^2\sin(\xi)+W\cos(\xi).
\end{aligned}
\end{equation*}
To determine $\rho_3$ it is simplest to check that $d(dC\wedge d\xi)=0$. Indeed, 
\begin{equation*}
\begin{aligned}
d(dC\wedge d\xi)&=d(\rho\,\omega^1\wedge\omega^2+\xi_3dC\wedge\omega^3),\\
0&=(\rho_3-2A_3\rho)\Omega+d\xi_3\wedge dC\wedge\omega^3-\xi_3dC\wedge d\omega^3.
\end{aligned}
\end{equation*}
However,
\[
d\xi_3=2A_3d\xi+\partiald{\xi_3}{C}dC
\]
and therefore
\[
0=(\rho_3-4A_3\rho+\xi_3B_2\sin(\xi)-\xi_3B_1\cos(\xi))\Omega
\]
so we conclude that 
\[
\rho_3=4A_3\rho+\xi_3B_1\cos(\xi)-\xi_3B_2\sin(\xi).
\]
This completely determines $d\rho$. Next we investigate $d^2\omega^3=0$. Indeed, 
\begin{equation*}
\begin{aligned}
d^2\omega^3&=\partiald{B_1}{C}dC\wedge \omega^2\wedge\omega^3+\partiald{B_1}{\xi}d\xi\wedge\omega^2\wedge\omega^3+
\partiald{B_1}{\rho}d\rho\wedge\omega^2\wedge\omega^3\\
&+\partiald{B_2}{C}dC\wedge \omega^1\wedge\omega^3+\partiald{B_2}{\xi}d\xi\wedge\omega^1\wedge\omega^3+
\partiald{B_2}{\rho}d\rho\wedge\omega^1\wedge\omega^3.
\end{aligned}
\end{equation*}
Conveniently, this results in a linear equation in $W$
\begin{equation*}
\begin{aligned}
0&=\cos(\xi)\left(\partiald{B_1}{C}+\rho\partiald{B_2}{\xi}-\rho^2\partiald{B_1}{\rho}\right)-\sin(\xi)\left(\partiald{B_2}{C}+\rho\partiald{B_1}{\xi}-\rho^2\partiald{B_2}{\rho}\right)\\
&-W\left(\cos(\xi)\partiald{B_2}{\rho}+\sin(\xi)\partiald{B_1}{\rho}\right).
\end{aligned}
\end{equation*}
We can now determine $W$ explicitly in terms of $C,\xi$ and $\rho$. 
\[
W=\frac{\cos(\xi)\left(\partiald{B_1}{C}+\rho\partiald{B_2}{\xi}-\rho^2\partiald{B_1}{\rho}\right)-\sin(\xi)\left(\partiald{B_2}{C}+\rho\partiald{B_1}{\xi}-\rho^2\partiald{B_2}{\rho}\right)}{\cos(\xi)\partiald{B_2}{\rho}+\sin(\xi)\partiald{B_1}{\rho}}.
\]
It is important to remark that $W$ is quadratic in $\rho$ since each $B_1$ and $B_2$ are linear in $\rho$. 

To determine $\rho$, we proceed by checking $d(d\rho\wedge d\xi)=0$. This calculation becomes difficult to complete by hand. As such, we relegate the calculation to the computer algebra software MAPLE and use the \textbf{Cartan} Package (which was written by Yunliang Yu, and now improved and maintained by Jeanne Clelland and Tom Ivey). In part because of the presence of $W$, we are left with a cubic polynomial in $\rho$. Solving, and checking additional consistencies with $d(d\rho \wedge dC)=0$, we arrive at the unique expression 
\[
\rho=\frac{\sin(2\xi)}{\xi_3+\epsilon -1}.
\] 

For the case of $C_3\neq0$, let
\[
d\zeta=\zeta_1\omega^1+\zeta_2\omega^2+\zeta_3\omega^3.
\]
Then 
\begin{equation}
\begin{aligned}
d^2C_3&\equiv C_3(\delta (A_1\cos(\zeta)+\zeta_1\sin(\zeta))-A_2\sin(\zeta)+\zeta_2\cos(\zeta))\omega^1\wedge\omega^2\mod\{\omega^3\}
\end{aligned}
\end{equation}
meaning there is a function $W$ on $M$ such that
\begin{equation}
\begin{aligned}
\zeta_1&=W\cos(\zeta)+\delta A_2,\\
\zeta_2&=-\delta W\sin(\zeta)-\delta A_1,
\end{aligned}
\end{equation} 
as claimed. 
%%%%%%%%%%
\comment{
In fact, 
\begin{equation*}
\begin{aligned}
\partiald{B_1}{\rho}&=\frac{((C+A_3)\cos(\xi)-\sin(\xi))}{2(C^2-A_3^2+\epsilon )}\partiald{A_3}{\xi}\\
\partiald{B_2}{\rho}&=\frac{(-\epsilon \cos(\xi)+(C-A_3)\sin(\xi))}{2(C^2-A_3^2+\epsilon )}\partiald{A_3}{\xi}.
\end{aligned}
\end{equation*}
It now follows that the denominator of $W$ simplifies to
\[
\frac{\xi_3-\epsilon \cos(2\xi)-1}{2(C^2-A_3^2+\epsilon )}\,\partiald{A_3}{\xi}.
\]
Moreover, the numerator of $W$ may be simplified somewhat as well since
\[
\sin(\xi)\partiald{B_2}{\rho}-\cos(\xi)\partiald{B_1}{\rho}=-\frac{\,2A_3+\epsilon \sin(2\xi)}{2(C^2-A_3^2+\epsilon )}\,\partiald{A_3}{\xi}.
\]
Additionally, let us also use the identity
\[
C^2-A_3^2+\epsilon =\xi_3(\xi_3+\epsilon -1)
\]
so that a simpler expression for $W$ is
\begin{equation*}
\begin{aligned}
W&=\frac{2\xi_3(\xi_3+\epsilon -1)}{(\xi_3-\epsilon \cos(2\xi)-1)(1-\epsilon -2\xi_3)}\left(\cos(\xi)\left(\partiald{B_1}{C}+\rho\partiald{B_2}{\xi}\right)-\sin(\xi)\left(\partiald{B_2}{C}+\rho\partiald{B_1}{\xi}\right)\right)\\
&-\rho^2\frac{\,2A_3+\epsilon \sin(2\xi)}{\xi_3-\epsilon \cos(2\xi)-1}
\end{aligned}
\end{equation*}
Finally, in order to find an explicit formula for $\rho$ in terms of $C$ and $\xi$, we investigate $d(d\rho\wedge d\xi_3)=0$ conditions. To that end, let us calculate each $d\rho_i$ for $i=1,2,3$,
\begin{equation*}
\begin{aligned}
d\rho_1&=-2\rho\cos(\xi)d\rho+(\rho^2\sin(\xi)-W\cos(\xi))\,d\xi-\sin(\xi)\,dW,\\
d\rho_2&=-2\rho\sin(\xi)d\rho-(\rho^2\cos(\xi)+W\sin(\xi))\,d\xi+\cos(\xi)\,dW,\\
d\rho_3&=4\rho\,dA_3+4A_3d\rho+(B_1\cos(\xi)-B_2\sin(\xi))d\xi_3-\xi_3(B_1\sin(\xi)+B_2\cos(\xi))d\xi\\
&+\xi_3(\cos(\xi)dB_1-\sin(\xi)dB_2).
\end{aligned}
\end{equation*}
Notice that 
\begin{equation*}
\begin{aligned}
\cos(\xi)d\rho_1+\sin(\xi)d\rho_2&=-2\rho\,d\rho-W\,d\xi,\\
-\sin(\xi)d\rho_1+\cos(\xi)d\rho_2&=dW-\rho^2d\xi.
\end{aligned}
\end{equation*}
In particular 
\[
d\rho_1\wedge d\rho_2=-2\rho\,d\rho\wedge dW+2\rho^3d\rho \wedge d\xi-Wd\xi\wedge dW
\]
It follows that 
\begin{equation*}
\begin{aligned}
dW&=\rho^2d\xi-\sin(\xi)d\rho_1+\cos(\xi)d\rho_2,\\
d\rho&=-\frac{W}{2\rho}d\xi-\frac{\cos(\xi)}{2\rho}d\rho_1-\frac{\sin(\xi)}{2\rho}d\rho_2.
\end{aligned}
\end{equation*}
To calculate $d(d\rho\wedge d\xi)$ we start with
\begin{equation*}
\begin{aligned}
d\rho\wedge d\xi&=(\rho_1\rho\cos(\xi)+\rho_2\rho\sin(\xi))\omega^1\wedge \omega^2-(\rho_3\rho\sin(\xi)+\rho_1\xi_3)\omega^3\wedge \omega^1\\
&+(\rho_2\xi_3-\rho_3\rho\cos(\xi))\omega^2\wedge \omega^3,\\
d\rho\wedge d\xi&=-\rho^3\omega^1\wedge\omega^2-((\rho_3\rho-W\xi_3)\sin(\xi)-\rho^2\xi_3\cos(\xi))\omega^3\wedge\omega^1\\
&+((W\xi_3+\rho_3\rho)\cos(\xi)-\rho^2\xi_3\sin(\xi))\omega^2\wedge \omega^3
\end{aligned}
\end{equation*}
}
%%%%%%%%%%%
\end{proof}
We end this section with two propositions and two examples of the third case of $C$ non-constant. The proposition below was noted in \cite{PerroneBiContactHyp} but we see this arise naturally in the context of studying bi-contactomorphisms. Additionally, the construction in \cite{PerroneBiContactHyp} only yields the case of $E=0$. Let us also briefly give a convenient definition. A \textbf{symplectic quadratic} on a symplectic manifold is simply a pair of symplectic forms $\theta^1,\theta^2$ such that $\theta_a=a_1\theta^1+a_2\theta^2$ has the property that $\theta_a\wedge \theta_a=\mcal{P}_C(a)\Omega$ where $\mcal{P}_C(a)=1$ for some function $C$. As with bi-contact structures, it may be that a hyperbola defined by $\mcal{P}_C(a)$ may force the orientation of the volume form for $\theta^1$ and $\theta^2$ to oppose one another. 
\begin{prop}\label{symp-prop}
Given a bi-contact structure $(\omega^1,\omega^2)$ on some $M$, if $dC\wedge\omega^1\wedge\omega^2=C_3\Omega\neq0$ but $dC_3\wedge \omega^3=0$ on $M$ then there is a manifold $\widetilde{M}$ containing $M$ as a submanifold such that $\widetilde{M}$ admits a symplectic quadratic.
\end{prop}
\begin{proof}
From the structure equations \eqref{symp-struct-eqs} we see that 
\begin{equation}
\begin{aligned}
d\omega^1\wedge d\omega^1&= 2\omega^1\wedge\omega^2\wedge\omega^3\wedge\omega^4,\\
d\omega^2\wedge d\omega^2&= -2\epsilon \omega^1\wedge\omega^2\wedge\omega^3\wedge\omega^4.
\end{aligned}
\end{equation}
In fact, if $\theta^1=d\omega^1$ and $\theta^2=d\omega^2$ then the 2-form $\theta_a=a_1\theta^1+a_2\theta^2$ is a symplectic quadratic. 
\end{proof}
Once again, we find that the obstruction to the existence of a taut contact circle or equilateral hyperbola is also an obstruction for $\theta^1$ and $\theta^2$ to form a symplectic couple or symplectic pair \cite{GeigesSymplecticCouple}(i.e. $\epsilon=\pm1$). In particular the invariant $C$ arises in the form of $\theta^1\wedge \theta^2=2C\omega^1\wedge\omega^2\wedge\omega^3\wedge \omega^4$. 
Let us give an example of the case $E=0$. Reconsider example \ref{hyp-ex-C3-nonzero}. Relabel the 1-forms $\omega^1$ and $\omega^2$ as $e^{-s}\omega^1+e^{-s}(ds+s(1-z)C(z)^{-1}dz)$ and $e^{-s}\omega^2+e^{-s}(ds+s(z+\epsilon )C(z)^{-1}dz)$ and introduce $\omega^4=ds$. Then
\begin{equation}
\begin{aligned}
d\omega^1&=\omega^1\wedge\omega^4-z\omega^1\wedge\omega^3+\omega^2\wedge\omega^3,\\
d\omega^2&=\omega^2\wedge\omega^4+\omega^1\wedge\omega^3-z\omega^2\wedge\omega^4,\\
d\omega^3&=0,\\
d\omega^4&=0.
\end{aligned}
\end{equation}
Notice that each symplectic form $\theta^i=d\omega^i$, $i=1,2$ share the family of Lagrangian submanifolds determined by $(x,y,z_0,s_0)$ for $z_0$ and $s_0$ constants.

Now we consider an example where the invariant $E\neq0$. In general, such examples are locally foliated by contact 3-manifolds that are transverse to $M$. 
\begin{ex}
Let $M$ and $\widetilde{M}$ be as in Proposition \ref{symp-prop} and let $(x,y,z,w)$ be local coordinates for $\widetilde{M}$. Moreover, we have taken $C(z)=\tan(z)$ and we assume that $-\pi/2<z<\pi/2$. Matters being so, the following 1-forms satisfy the structure equations \eqref{symp-struct-eqs} with $E=y^{-1}e^{2w-y(x+z)}$. 
\begin{equation}
\begin{aligned}
\omega^1&=ze^{-w}(dx+dz)+ye^{y(x+z)-w}dy,\\
\omega^2&=-(z\tan(z)+1)e^{-w}(dx+dz)-y\tan(z)e^{y(x+z)-w}dy,\\
\omega^3&=dz,\\
\omega^4&=dw-y(dx+dz).
\end{aligned}
\end{equation}
It is easy to verify that $E$ is as claimed since $d\omega^4=-dy\wedge(dx+dz)$. 
\end{ex}
\begin{prop}\label{E-const-zero}
If $E$ is constant then $E=0$. 
\end{prop}
\begin{proof}
If $E$ is constant then $d^2\omega^4=E(d\omega^1\wedge \omega^2-\omega^1\wedge d\omega^2)$ but $d\omega^1\wedge \omega^2-\omega^1\wedge d\omega^2=2\omega^1\wedge\omega^2\wedge\omega^4$ thus forcing $E=0$ to satisfy $d^2\omega^4=0$. 
\end{proof}
\subsection{Constant C}\label{constant C}

The contactomorphism group for a contact structure on a 3-manifold is infinite dimensional, but parameterized by a single function of 3 variables. That is, the space of contactomorphisms is in one-to-one correspondence with smooth functions on $M$. We now prove the following theorem which yields a similar result in case the invariant function $C$ is constant. 

\begin{thm}\label{constant-C-equiv}
All analytic bi-contact structures with the same non-zero constant $C$ and $\epsilon $ are locally equivalent under bi-contactomorphisms. Moreover, the pseudo-group of bi-contactomorphisms for a given bi-contact structure is locally parameterized by two functions of one variable. 
\end{thm}
\begin{proof}
We will prolong the principal bundle associated to the non-involutive structure equations (\ref{Struct Eq 1}) to a new principal bundle which we denote $\mcal{B}^{(1)}_1$. We remark that our prolongation is not the full prolongation by all the freedom available from $\text{dim} A^{(1)}$ in the application of Cartan's test. Instead, since $\alpha$ is uniquely defined, it suffices to prolong only in the $\alpha$ direction and leave the $\beta_1$ and $\beta_2$ unchanged. Relabeling $\alpha=\omega^4$, we take $\pr{\omega}{1}$ as the new tautological 1-form with components $\omega^i$ for $1\leq i\leq 4$. 
To determine the structure equations, we need to determine $d\omega^4$ from $d^2\omega^1=d^2\omega^2=0$. Such calculations lead to the following: 
\begin{equation}
\begin{aligned}
d^2\omega^1 & = -d\omega^4 \wedge \omega^1 +\omega^4\wedge d\omega^1 +d(-C\omega^1+\omega^2)\wedge \omega^3 -(-C\omega^1+\omega^2)\wedge d\omega^3, \\
& = -d\omega^4\wedge \omega^1 + \omega^4 \wedge (-C\omega^1+\omega^2)\wedge \omega^3 +C\omega^4 \wedge \omega^1\wedge\omega^3 -\omega^4\wedge \omega^2 \wedge \omega^3\\
& +(-C\omega^1+\omega^2)\wedge(\beta_1 \wedge \omega^1 +\beta_2\wedge\omega^2),\\
& =  -(d\omega^4 +(\beta_1+C\beta_2)\wedge \omega^2) \wedge \omega^1 
\end{aligned}
\end{equation}
and 

\begin{equation}
\begin{aligned}
d^2\omega^2 & = -d\omega^4 \wedge \omega^2 +\omega^4\wedge d\omega^2 +d(\epsilon\omega^1+C\omega^2)\wedge \omega^3 -(\epsilon\omega^1+C\omega^2)\wedge d\omega^3, \\
& = -d\omega^4\wedge \omega^2 + \omega^4 \wedge (\epsilon\omega^1+C\omega^2)\wedge \omega^3 -\epsilon\omega^4 \wedge \omega^1\wedge\omega^3 -C\omega^4\wedge \omega^2 \wedge \omega^3\\
& +(\epsilon\omega^1+C\omega^2)\wedge(\beta_1 \wedge \omega^1 +\beta_2\wedge\omega^2),\\
& =  -(d\omega^4 +(-C\beta_1+\epsilon\beta_2)\wedge \omega^1) \wedge \omega^2
\end{aligned}
\end{equation}
which in turn imply that there exists some function $E$ such that 
\begin{equation}
d\omega ^4  =-(-C\beta_1+\epsilon\beta_2)\wedge\omega^1 - (\beta_1+C\beta_2)\wedge\omega^2+E\omega^1\wedge\omega^2. 
\end{equation}
This now yields the structure equations 
\begin{equation} 
d\omega^{(1)}=-\begin{bmatrix} 0 & 0 & 0 & 0 \\ 0 & 0 & 0 & 0 \\ \beta_1 & \beta_2 & 0 & 0 \\ -C\beta_1+\epsilon\beta_2 & \beta_1+C\beta_2 & 0 & 0  \end{bmatrix}\wedge \omega^{(1)} + \begin{bmatrix} \omega^1\wedge \omega^4 -C\omega^1\wedge \omega^3 +\omega^2\wedge \omega^3\\
\omega^2\wedge \omega^4 +\epsilon\omega^1\wedge\omega^3 + C\omega^2 \wedge \omega^3\\ 0 \\ E \omega^1\wedge \omega^2 \end{bmatrix}.  
\end{equation}
 Notice that the $E\,\omega^1\wedge\omega^2$ torsion term is actually entirely absorbable via 
\begin{equation}
\begin{aligned}
\beta_1 &\mapsto \beta_1 + \frac{1}{2}E\omega^1,\\
\beta_2 &\mapsto \beta_2 -\frac{1}{2}\epsilon\,E\omega^2,
\end{aligned}
\end{equation}
and therefore 
\begin{equation} 
d\omega^{(1)}=-\begin{bmatrix} 0 & 0 & 0 & 0 \\ 0 & 0 & 0 & 0 \\ \beta_1 & \beta_2 & 0 & 0 \\ -C\beta_1+\epsilon\beta_2 & \beta_1+C\beta_2 & 0 & 0  \end{bmatrix}\wedge \omega^{(1)} + \begin{bmatrix} \omega^1\wedge \omega^4 -C\omega^1\wedge \omega^3 +\omega^2\wedge \omega^3\\
\omega^2\wedge \omega^4 +\epsilon\omega^1\wedge\omega^3 + C\omega^2 \wedge \omega^3\\ 0 \\ 0 \end{bmatrix}.  
\end{equation}
Since all remaining torsion is constant, then we have no opportunity to normalize the group any further. As such, we need to test for involutivity. The Cartan characters are $s_1=2$ and $s_2=s_3=s_4=0$. As before, if we take $\beta_1'$ and $\beta_2'$ as different 1-forms on the fibers, we find again by Cartan's lemma that 
\begin{equation}
\begin{bmatrix}\beta_1\\ \beta_2\end{bmatrix}=\begin{bmatrix} \beta'_1\\ \beta'_2 \end{bmatrix}+\begin{bmatrix} h_{11} & h_{12}\\ h_{12} & h_{22} \end{bmatrix}\begin{bmatrix} \omega_1\\ \omega_2 \end{bmatrix}. 
\end{equation}
for some functions $h_{ij}$. However, we now have contribution from the $d\omega^4$ and substituting the above yields 
\begin{equation}
(-Ch_{12}+\epsilon h_{22}-h_{11}-C\,h_{12})\omega^1\wedge \omega^2=0
\end{equation}
and therefore we go from freedom in three variables for the pseudo-connection to only two and therefore $\text{dim}A^{(1)}=2$, where we have recylced the notation of $A$ as the new tableau. Thus, we have that $\text{dim}A^{(1)}=s_1+2\,s_2+3\,s_3+4\,s_4=2$ and hence we have involutivity with $s_1=2$ giving us the two functions of one variable function count for bicontactomorhisms. Thus, again by application of the Cartan-K\"ahler theorem we have equivalence in the analytic category.
\end{proof}
\begin{ex}\label{T^2xR ex}
A simple example of the constant $C$ case for hyperbolic bi-contact structures is the following: let $M=\mathbb{T}^2\times\mathbb{R}$ with coordinates $(\theta,\varphi,z)$ and consider 
\begin{equation*}
\begin{aligned}
\omega^1&=f(z)d\theta + g(z)d\varphi + (g(z)-2\,C\,f(z))dz,\\
\omega^2&=(2\,C\,f(z) - g(z))d\theta - f(z)d\varphi + g(z)dz,\\
\omega^3&=(f(z)g'(z)-f'(z)g(z))dz,
\end{aligned}
\end{equation*}
for any $f,g\in C^1(\mathbb{R})$ such that $f(z)^2-g(z)^2+2Cf(z)g(z)=1$. 
\end{ex}
It is straightforward to check that the constant $C$ is indeed the invariant $C$ and that $\omega^2\wedge d\omega^2=-\omega^1\wedge d\omega^1$. This demonstrates that $M$ admits infinitely many bi-contactomorphism classes parameterized by $\mathbb{R}$. Additionally, one may explicitly parameterize the family of hyperbolas determined by $f(z)^2-g(z)^2+2Cf(z)g(z)=1$ via 
\begin{equation}
\begin{bmatrix}f(z)\\ g(z)\end{bmatrix}=\begin{bmatrix} \cosh(\psi) & \sinh(\psi)\\ -\sinh(\psi) & \cosh(\psi) \end{bmatrix}^{-1}\begin{bmatrix}\cosh\left(\Psi(z)\right)\\ \sinh\left(\Psi(z)\right)\end{bmatrix}
\end{equation}
for some constant $\psi$ such that $C=\sinh(2\psi)$ and some smooth function $\Psi(z)$. Finally, we observe that the volume form is given by
\[
\Omega=-\cosh(2\psi)\frac{d\Psi}{dz}d\theta\wedge d\varphi \wedge dz
\]
and hence it is possible to choose $\Psi(z)$ so that the manifold has finite volume. 
%%%%%%%%%%%%%%%%%%%%%%%%%%%%%%%%%%%%%%%%%%%%%%%%%%%%%%%%%%%%%%%%%%%%%%%%%%%%%%%%%%

\subsection{\texorpdfstring{$(-\epsilon )$}{-epsilon}-Cartan Structures}
There is an additional special case introduced in \cite{GeigesTautCirc} and investigated some in \cite{PerroneBiContactHyp}. 
\begin{defn}
A $(-\epsilon )$-Cartan structure is a bi-contact structure such that 
\[
\omega^1\wedge d\omega^2=\omega^2 \wedge d\omega^1 =0.
\]
\end{defn}
Clearly, such structures must also have $C=0$. However, we have additional rigidity imposed on bi-contactomorphism equivalence, in particular, we will be forced to take the $\alpha$ 1-form in the structure equations \eqref{Struct Eq 1} to be basic, so $(-\epsilon)$ Cartan structures are not properly covered by Theorem \ref{constant-C-equiv}.  There is an immediate class of examples of $(-\epsilon)$-Cartan structures: the solder forms $(\theta^1,\theta^2)$ on the orthonormal frame bundle of a pseudo-Riemannain surface $\Sigma$ with metric $g_{-\epsilon}=(\underline{\theta}^1)^2-\epsilon(\underline{\theta}^2)^2$ where $\pi^* \underline{\theta}^i=\theta^i,\,\,\,i=1,2$ and $\pi: \mcal{F}\Sigma^2\to\Sigma^2$

 We attribute the following theorem to Perrone for his structure equations in item (iii) of theorem 7.1 from \cite{PerroneBiContactHyp} and also to Geiges and Gonzalo who also made the observation in \cite{GeigesTautCirc}. Though we have additionally uncovered the equivalence under bi-contactomorphisms. 
\begin{thm}
Let $(M,\omega^1,\omega^2)$ be a $(-\epsilon)$-Cartan structure. Then there exists a pseudo-Riemannian surface $(\Sigma^2,g_{-\epsilon})$ where $g_{-\epsilon}=(\underline{\theta}^1)^2-\epsilon(\underline{\theta}^2)^2$ such that $(M,\omega^1,\omega^2)$ is bi-contactomorphic to $(\mcal{F}\Sigma^2,\theta^1,\theta^2)$ and vice versa i.e. given a surface there exists a corresponding bi-contactomorphism class that is a $(-\epsilon)$-Cartan structure. 
\end{thm}
\begin{proof}
The proof follows by continuing Cartan's method of equivalence. We start from the structure equations \eqref{Struct Eq 1} where we have that $C=0$. The $(-\epsilon )$-Cartan structure requirement means that we must impose 
\[
\omega^1 \wedge d\omega^2 = \omega^2 \wedge d\omega^1=0
\]
on the structure equations. Doing so yields 
\[
\alpha\wedge \omega^1\wedge \omega^2 =0
\]
meaning that $\alpha$ must in fact be basic and depend solely on $\omega^1$ and $\omega^2$. As such we set $\alpha=A_1\omega^1+A_2\omega^2$ for some $A_1$ and $A_2$ on the coframe bundle. We now have the following structure equations 
\begin{equation}
 d\begin{bmatrix}
 \omega^1\\ \omega^2\\ \omega^3
 \end{bmatrix}
 =
  -\begin{bmatrix}
   0 & 0 & 0\\
   0 & 0 & 0\\
   \beta_1 & \beta_2 & 0\\
  \end{bmatrix}\wedge
 \begin{bmatrix}
 \omega^1\\ \omega^2\\ \omega^3
 \end{bmatrix}+
\begin{bmatrix}
A_2\omega^1\wedge\omega^2+ \omega^2\wedge\omega^3\\  -A_1\omega^1\wedge\omega^2 +\epsilon \omega^1\wedge\omega^3\\ 0
\end{bmatrix}.
\end{equation}
Notice that 
\begin{equation*}
\begin{aligned}
d^2\omega^1&=(dA_2-A_1\omega^3+\beta_1)\wedge \omega^1\wedge \omega^2,\\
d^2\omega^2&=-(dA_1-\epsilon A_2\omega^3+\epsilon \beta_2)\wedge \omega^1\wedge \omega^2.
\end{aligned}
\end{equation*}
Thus we can normalize the torsion terms $A_1$ and $A_2$ to zero. As such, $\beta_1$ and $\beta_2$ become basic, but in such way that $\beta_i \equiv 0\mod \{\omega^1,\omega^2\}$ so that $d^2\omega^1=d^2\omega^2=0$. As such, if we take
\begin{equation*}
\begin{aligned}
\beta_1&=B_1\omega^1+\frac{1}{2}\left(B_3+K\right)\omega^2,\\
\beta_2&=\frac{1}{2}\left(K-B_3\right)\omega^1+B_2\omega^2,
\end{aligned}
\end{equation*}
where $B_i$, $1\leq i\leq k$ and $K$ are new torsion terms. This results in the following structure equations 
\begin{equation*}
\begin{aligned}
d\omega^1&=\omega^2 \wedge \omega^3,\\
d\omega^2&=\epsilon \omega^1\wedge \omega^3,\\
d\omega^3&=K\omega^1\wedge\omega^2. 
\end{aligned}
\end{equation*}
Moreover, in order that $d^2\omega^3=0$ we must have $dK\wedge \omega^1\wedge \omega^2=0$ everywhere. These are precisely the structure equations for the coframe bundle of a Riemannian or Lorentzian surface (depending on the sign of $\epsilon $) where $\omega^1$ and $\omega^2$ make up the orthonormal coframe for a surface $\Sigma$ and $\omega^3$ plays the role of the connection 1-form and $K$ is the Gaussian curvature. 
\end{proof}
In \cite{PerroneBiContactHyp} Perrone mentions this relation but only in the context of Webster curvature (which reduces to Gaussian curvature in special cases), yielding some nice results about associated contact metrics. Moreover, the structure equations Perrone presents have a factor of two that can be eliminated to resemble the above structure equations by scaling the contact 1-forms by one-half.

\section{Two Local Normal Forms}\label{sec-norm-form}
In this section we present some special circumstances of the non-zero $C_3$ case. In \cite{Jackman-Bi-Contact}, the author finds a number of several other normal forms which are simpler in nature. The author of \cite{Jackman-Bi-Contact} and the first author of this paper are working on extending these results and other general facets of bi-contact geometry. First, we remark that several combinations of the invariant functions may be used to form local invariant coordinate systems, implying the existence of local normal forms. We prove the existence of at least one invariant coordinate system for a bi-contact structure but are unable to find a normal form due to the complexity of the associated PDEs; however, when we consider a further special case ($A_1=A_2=0$ identically on $M$) we are able recover a local normal form.
Secondly, in the special case that $dC\wedge dC_3=0$ with $C_3\neq0$ and $B_1=B_2=0$, we recover a normal form on the 4-dimensional bundle $\wt{M}=\mcal{B}_2$. 

\begin{prop}\label{C-coords}
Let $(\omega^1,\omega^2,\omega^3)$ constitute a coframing on a 3-manifold $M$ that has the structure equations \eqref{struct-eqs 2}. On any open $U\subset M$ such that $C_3\neq0$ there is an invariant coordinate system defined by $(x,y,z)=(C,C_3,C_{33})$ for at least one of the possibilities $\epsilon =\pm1$. Generically, such an invariant coordinate system exists for both cases of $\epsilon =\pm1$. 
\end{prop}
\begin{proof}
We need to find an expression for $dC_{33}$, which we achieve by checking the conditions imposed by $d^2C_3=0$. By the structure equations from Theorem \ref{non-const-C-struct} and Proposition \ref{der-conditions} we find that 
\begin{equation}
\begin{aligned}
dC_{33}&=(C_3(\zeta_3-\delta \epsilon )\cos(\zeta)+(2C_{33}+C_3C-C_3A_3)\sin(\zeta))\omega^1\\
\,&+(-C_3(\zeta_3+\delta )\sin(\zeta)+\delta (2C_{33}-C_3C-C_3A_3)\cos(\zeta))\omega^2+C_{333}\omega^3.
\end{aligned}
\end{equation}
We also have
\begin{equation}
\begin{aligned}
dC\wedge dC_3&=C_3^2(\sin(\zeta)\omega^1\wedge\omega^3+\delta \cos(\zeta)\omega^2\wedge \omega^3),
\end{aligned}
\end{equation}
and it now follows that 
\begin{equation}
\begin{aligned}
\Omega_C^{\epsilon }&= -C_{3}^3(1-(1+\epsilon )\cos^2(\zeta)+2\delta C\sin(\zeta)\cos(\zeta)+\delta \zeta_3)\omega^1\wedge\omega^2\wedge\omega^3,
\end{aligned}
\end{equation}
where $\Omega_C^{\epsilon }=dC\wedge dC_3\wedge dC_{33}$. Thus, aside from $C_3\neq 0$, the function $\zeta$ must not solve the following PDE
\begin{equation}\label{zeta-PDE}
\zeta_3=-\delta +\delta (1+\epsilon )\cos^2(\zeta)-2C\sin(\zeta)\cos(\zeta),
\end{equation}
otherwise, the 3-form $\Omega_C^{\epsilon }$ is zero when this is the case. Assume $\zeta$ solves \eqref{zeta-PDE} on some open $U\subset M$. Then the 3-form $\Omega_C^{-\epsilon }$ has the expression
\begin{equation}
\Omega_C^{-\epsilon }=-2C_3^3\cos^2(\zeta)\omega^1\wedge\omega^2\wedge\omega^3. 
\end{equation}
Thus, so long as $C_3\neq0$ and $\zeta\neq \frac{(2k+1)\pi}{2}$ the volume form $\Omega^{-\epsilon }_C$ is nonzero. In fact, we can go further. If $\zeta$ satisfies the PDE \eqref{zeta-PDE}, and $p\in M$ any point where $\zeta_3\big|_p=0$, then $\zeta_p=\zeta\big|_p$ and $C_p=C\big|_p$ must satisfy
\begin{equation}
-\delta +\delta (1+\epsilon )\cos^2(\zeta_p)=2C_p\sin(\zeta_p)\cos(\zeta_p),
\end{equation}
equivalently
\begin{equation}
\epsilon \cos^2(\zeta_p)-\sin^2(\zeta_p)=\delta C_p\sin(2\zeta_p).
\end{equation}
For $\epsilon =1$ this yields that 
\begin{equation}
\zeta_p=\frac{1}{2}\text{arccot}(\delta  C_p),
\end{equation}
which means that $\zeta_p$ can only ever take values in the open intervals $(n\pi/2,(n+1)\pi/2)$ for $n\in \mathbb{Z}$. In case $\epsilon =-1$, we have instead 
\begin{equation}
\zeta_p=\frac{1}{2}\text{arccsc}(-\delta C_p).
\end{equation}
Thus, $\zeta_p$ can only take values in the intervals $[n\pi/2-\pi/4, n\pi/2) \cup (n\pi/2,n\pi/2+\pi/4]$. Therefore, for $C_3\neq0$ on any open $U\subset M$, whenever $\Omega_C^{\epsilon }=0$ then $\Omega_C^{-\epsilon }$ is never zero. 
\end{proof}

\begin{remark}
The PDE \eqref{zeta-PDE} may be transformed into the equivalent PDEs
\begin{equation}\label{alt-zeta-PDE}
\begin{aligned}
\partiald{H}{x}&=\frac{1}{y}(1-x\sinh(2H)),\quad \epsilon =1\\
\partiald{G}{x}&=-\frac{1}{y}(1+x\sin(2G)), \quad \epsilon =-1,
\end{aligned}
\end{equation}
where $x=C$ and $y=C_3$. In particular, the PDE for function $H$ may be transformed to \textit{almost} the PDE for $G$. Indeed, let $\widehat{H}=i\,H$ and $x\mapsto -i\,\widehat{x}$. Then the PDE for $H$ becomes 
\[
\partiald{\widehat{H}}{\widehat{x}}=\frac{1}{y}(1+\widehat{x}\sin(2\widehat{H})).
\] 
It is interesting to note that, although the PDEs \eqref{alt-zeta-PDE} are nearly related by complex rotations, the square of the derivatives may be transformed to one another, at the cost of swapping branches. 
\end{remark}

Now we restrict to a special case and find a local normal form. As a consequence, we manage to uncover an additional bi-contact structure. First, an important lemma. 

\begin{lem}\label{C-zeta}
Let $M$ be a 3-manifold with bi-contact structure satisfying the structure equations \eqref{struct-eqs 2}. If $A_1=A_2=0$ on an open subset $U$ of $M$ then $dC\wedge d\zeta =0$ on $U$. Moreover, 
\begin{enumerate}
\item If $\epsilon =-1$ then $C=-\csc(2\zeta)$ and 
\item if $\epsilon =1$ then $C=\cot(2\zeta)$. 
\end{enumerate}
\end{lem}
\begin{proof}
The proof is an exercise in checking $d^2=0$ conditions. Moving forward with our computation, we find that since $A_1=A_2=0$
\[
\zeta_1\sin(\zeta)+\zeta_2\cos(\zeta)=0,
\]
where subscripts on $\zeta$ denote covariant derivatives. This implies the existence of a function $W$ such that 
\[
\zeta_1=-W\cos(\zeta)\text{ and }\zeta_2=W\sin(\zeta). 
\]
Next, we assign some covariant derivatives of $A_3$ by inspecting $d^2\omega^1$ and $d^2\omega^2$. This yields
\begin{equation*}
\begin{aligned}
(A_{3})_1&=\epsilon \cos(\zeta)-(A_3+C)\sin(\zeta),\\
(A_{3})_2&=\sin(\zeta)-(A_3-C)\cos(\zeta).
\end{aligned}
\end{equation*}
We need not assign a covariant derivative for $A_3$ in the direction of $\be_3$ at this time. Indeed, it is sufficient to compute $d^2A_3 \mod \{\omega^3\}$ and so we find
\[
d^2A_3 \equiv (2(W+1)\cos(\zeta)\sin(\zeta)C-W((\epsilon +1)\cos^2(\zeta)-\epsilon )-\epsilon \cos^2(\zeta)+\sin^2(\zeta)))\omega^1\wedge \omega^2\,\text{mod}\{\omega^3\}.
\]
This forces the above coefficient of $\omega^1\wedge \omega^2$ to vanish. We may use this to solve for the invariant function $C$ and indeed this yields 
\begin{equation}\label{C-zeta-eq-1}
C=\frac{W((\epsilon +1)\cos^2(\zeta)-\epsilon )+\epsilon \cos^2(\zeta)-\sin^2(\zeta))}{2(W+1)\cos(\zeta)\sin(\zeta)}.
\end{equation}
At this point, it is now prudent to split into cases. Let us first handle the case that $\epsilon =1$. Then our expression for $C$ given by equation \eqref{C-zeta-eq-1} becomes
\[
C=\cot(2\zeta).
\]
We now have the form claimed in the Lemma. 
The case that $\epsilon =-1$ requires marginally more effort. Indeed, we find that 
\[
C=\frac{W-1}{(W+1)\sin(2\zeta)}, 
\]
and so to prove our claim we must demonstrate that $W=0$ on $M$. We calculate $dC$ and find that 
\[
dC\equiv \frac{2\sin(\zeta)W_1+W(W^2-1)\cos(2\zeta)}{(W+1)^2\sin(2\zeta)\sin(\zeta)}\omega^1+\frac{2\cos(\zeta)W_2-W(W^2-1)\cos(2\zeta)}{(W+1)^2\sin(2\zeta)\cos(\zeta)}\omega^2\,\text{mod}\{\omega^3\},
\]
where $dW\equiv W_1\omega^1+W_2\omega^2\mod\{\omega^3\}$. We recall that in this adapted coframing $dC=C_3\omega^3$ and therefore it must be the case the the above coefficients of $\omega^1$ and $\omega^2$ vanish. Indeed, this yields
\begin{equation*}
\begin{aligned}
W_1&=-\frac{W(W^2-1)\cos(2\zeta)}{2\sin(\zeta)},\\
W_2&=\frac{W(W^2-1)\cos(2\zeta)}{2\cos(\zeta)}. 
\end{aligned}
\end{equation*}
Now we investigate the conditions forced upon us by $d^2\zeta=0$. Indeed, since we now have that 
\[
d\zeta \equiv -W\cos(\zeta)\omega^1+W\sin(\zeta)\omega^2\mod\{\omega^3\}
\]
then
\begin{equation*}
\begin{aligned}
d^2\zeta &\equiv (\cos(\zeta)W_2+\sin(\zeta)W_1-W\sin(\zeta)\zeta_2+W\cos(\zeta)\zeta_1)\omega^1\wedge\omega^2\mod\{\omega^3\},\\
\,&\equiv -W^2\omega^1\wedge\omega^2 \mod\{\omega^3\}.
\end{aligned}
\end{equation*}
Thus, we must deduce that $W=0$ on $M$. Hence when $\epsilon =-1$ we have that $C=-\csc(2\zeta)$ as claimed.
\end{proof}
There is an immediate corollary of Lemma \ref{C-zeta}. 
\begin{cor}
If $(\omega^1,\omega^2)$ is a bi-contact structure a 3-manifold $M$ with $C_3\neq 0$ and $A_1=A_2=0$ identically on $M$ then $(\omega^1,\omega^2)$ is never elliptic. 
\end{cor}
\begin{thm}\label{normal-form}
Let $(\omega^1,\omega^2,\omega^3)$ be any coframing on a 3-manifold $M$ that satisfies the structure equations \eqref{struct-eqs 2} and let $\zeta_3,$ and $\zeta_{33}$ be the derivatives of $\zeta$ and $\zeta_3$ respectively in the direction dual to $\omega^3$. If $A_1=A_2=0$ identically on an open set $U$ of $M$ then there exists a local invariant coordinate system $(x,y,z)=(\zeta,\zeta_3,\zeta_{33})$ where $\zeta\in(0,\pi/4)$ and $\zeta_3\neq 0$, and arbitrary functions $f(x)\in C^3(\mathbb{R})$ and $g(x)\in C^2(\mathbb{R})$ such that the coframing $(\omega^1,\omega^2,\omega^3)$ has the form 
\begin{equation*}
\begin{aligned}
\omega^1&=\cos(x)\eta^1+\sin(x)\eta^2,\\
\omega^2&=-\sin(x)\eta^1+\cos(x)\eta^2,\\
\omega^3&=\frac{1}{y}dx,
\end{aligned}
\end{equation*}
where $\eta^1$ and $\eta^2$ are 1-forms given by
\begin{equation*}
\begin{aligned}
\eta^1&=\frac{1}{y^2}(N_1dx+N_2dy-dz),\\
\eta^2&=\frac{z}{y^2}\,dx-\frac{1}{y}dy.
\end{aligned}
\end{equation*}
The functions $N_1$ and $N_2$ are dependent on $\epsilon $.  
\begin{enumerate}
\item If $\epsilon =1$ then
\begin{equation*}
\begin{aligned}
N_1&=y^2g(x)+2y(f(x)-2\cot(2x))\csc(2x)-(f(x)^2+f'(x)+1)y^2\ln(y)-zf(x),\\
N_2&=\frac{2z-y^2f(x)-2y\csc(2x)}{y},
\end{aligned}
\end{equation*}
where $f'(x)$ denotes the derivative of $f(x)$. 
\item else if $\epsilon =-1$ then
\begin{equation*}
\begin{aligned}
N_1&=y^2g(x) -2y\cot(2x)f(x)+2y+4y\csc^2(2x)-(f(x)^2+f'(x)+1)y^2\ln(y)-zf(x),\\
N_2&=\frac{2z-y^2f(x)+y\cot(2x)}{y}.
\end{aligned}
\end{equation*}
\end{enumerate}	
\end{thm}
\begin{proof}
We first demonstrate that $\zeta,\zeta_3$ and $\zeta_{33}$ yield a local invariant coordinate system. Now by Lemma \ref{C-zeta} and the structure equations \eqref{struct-eqs 2} we find that $d\zeta=\zeta_3\,\omega^3$ regardless of either $\epsilon =\pm1$ case from Lemma \ref{C-zeta}. Indeed, this turns out to be true for the following computations below, and we find that only $\epsilon $ itself contributes. Differentiating again we find
\[
d\zeta_3=-\zeta_3\sin(\zeta)\omega^1-\zeta_3\cos(\zeta)\omega^2+\zeta_{33}\omega^3,
\]
and differentiating yet again we discover
\begin{equation*}
\begin{aligned}
d\zeta_{33}&=\left(-\zeta_3\left(\zeta_3+\frac{1-\epsilon }{2}\right)\cos(\zeta)+(A_3\zeta_3-2\zeta_{33})\sin(\zeta)+\frac{1}{2}\zeta_3\sec(\zeta)\right)\omega^1\\
\,&+\left((A_3\zeta_3-2\zeta_{33})\cos(\zeta)+\zeta_3\left(\zeta_3+\frac{1-\epsilon }{2}\right)\sin(\zeta)+\epsilon \frac{1}{2}\zeta_3\csc(\zeta) \right)\omega^2+\zeta_{333}\omega^3.
\end{aligned}
\end{equation*}
Thus, we now see that 
\[
d\zeta \wedge d\zeta_3\wedge d\zeta_{33} = -\zeta_3^4\,\omega^1\wedge \omega^2\wedge \omega^3. 
\]
Therefore, away from subsets of $M$ where $\zeta_3=0$, we find that $(\zeta,\zeta_3,\zeta_{33})$ forms a local coordinate system. Let us label such a coordinate system as $(x,y,z)$ where $x=\zeta,y=\zeta_3,$ and $z=\zeta_{33}$. We now wish to find expressions of $\omega^1,\omega^2$ and $\omega^3$ in terms of $dx,dy,$ and $dz$. Let us now specialize to the case of $\epsilon =-1$ and invert the relations for $dx$, $dy$ and $dz$ above. 
\begin{equation*}
\begin{aligned}
\omega^1&= \frac{1}{2y^4}\left((2yz\varphi+2y\psi-4z^2)\cos(x)+2y^2z\sin(x)-\frac{yz\cos(2x)}{\sin(x)}\right)dx\\
\,&-\frac{1}{2y^3}\left((2y\varphi-4z)\cos(x)+2y^2\sin(x)-\frac{y\cos(2x)}{\sin(x)} \right)dy-\frac{\cos(x)}{y^2}dz,\\
\omega^2&=-\frac{1}{2y^4}\left((2yz\varphi+2y\psi-4z^2)\sin(x)-2y^2z\cos(x)-\frac{yz\cos(2x)}{\cos(x)}\right)dx\\
\,&+\frac{1}{2y^3}\left((2y\varphi-4z)\sin(x)-2y^2\cos(x)-\frac{y\cos(2x)}{\cos(x)} \right)dy+\frac{\sin(x)}{y^2}dz,\\
\omega^3&=\frac{1}{y}dx, 
\end{aligned}
\end{equation*}
where $\varphi(x,y,z)=A_3$ and $\psi(x,y,z)=\zeta_{333}$. We now arrange $\omega^1$ and $\omega^2$ in terms of $\eta^1$ and $\eta^2$ as claimed and then we will recover the functions $\varphi$ and $\psi$. Notice that 
\begin{equation*}
\begin{aligned}
\omega^1&= \left(\frac{yz\varphi+y\psi-2z^2-yz\cot(2x)}{y^2}dx-\frac{y\varphi-2z-y\cot(2x)}{y}dy-dz\right)\frac{\cos(x)}{y^2}\\
\,&+\left(zdx-ydy\right)\frac{\sin(x)}{y^2}\\
\omega^2&=-\left(\frac{yz\varphi+y\psi-2z^2-yz\cot(2x)}{y^2}dx-\frac{y\varphi-2z-y\cot(2x)}{y}dy-dz\right)\frac{\sin(x)}{y^2}\\
\,&+\left(zdx-ydy\right)\frac{\cos(x)}{y^2}\\
\end{aligned}
\end{equation*}
Let
\begin{equation*}
\begin{aligned}
\eta^1&= \frac{yz\varphi+y\psi-2z^2-yz\cot(2x)}{y^4}dx-\frac{y\varphi-2z+y\cot(2x)}{y^3}dy-\frac{1}{y^2}dz,\\
\eta^2&=\frac{z}{y^2}dx-\frac{1}{y}dy.\\
\end{aligned}
\end{equation*}
Finally, we need to recover the functions $\varphi$ and $\psi$ in order to acquire the claimed form in the Theorem. Indeed, we may do so by checking $d^2=0$ conditions and solving the resulting PDEs. The $\omega^3$ 1-form does not contribute to this effort. Suppressing some calculation details, we find that $\sin(x)d^2\omega^1+\cos(x)d^2\omega^2=-\varphi_z/y^2$ and as such $\varphi$ must have no dependence on $z$. Using this fact, $\sin(x)d^2\omega^1$ reduces to
\[
\sin(x)d^2\omega^1 = \cos(2x)+(\varphi-y\varphi_y)\sin(2x). 
\]
It follows that $\varphi(x,y)$ has the form
\[
\varphi(x,y)=yf(x)-\cot(2x),
\]
where $f(x)$ is an arbitrary differentiable function. 
To determine $\psi(x,y,z)$, we simply compute $d\omega^1$ and $d\omega^2$ in terms of $\omega^1,\omega^2,$ and $\omega^3$ and solve the PDE conditions that are forced upon $\psi(x,y,z)$ in order to satisfy the structure equations. Again suppressing details, we find that 
\[
\omega^2\wedge d\omega^1 - \omega^1\wedge d\omega^2 = \psi_z-\frac{4z}{y}
\]
and therefore $\psi(x,y,z)=-2z\varphi(x,y)+2z^2/y+\xi(x,y)$ for some $\xi(x,y)$ since $\omega^2\wedge d\omega^1 - \omega^1\wedge d\omega^2 = -2\varphi(x,y) \Omega$ by the structure equations. The structure equations also yield the identity $\omega^1\wedge d\omega^1-\omega^2\wedge d\omega^2 = 0$ which imposes the condition
\begin{equation*}
\begin{aligned}
0&=\frac{3\xi(x,y)\cos(2x)}{y^2}-\frac{\xi_y(x,y)\cos(2x)}{y}+\cot(2x)(f(x)\cos(2x)-2(1+\csc(2x))\\
\,&-(f(x)^2+f'(x)+1)y\cos(2x)
\end{aligned}
\end{equation*}
with solution given by 
\[
\xi(x,y)=y^3g(x)-(f(x)^2+f'(x)+1)y^3\ln(y)+2(1+\csc^2(2x)-f(x)\cot(2x))y^2
\]
for some $g(x)$. At this point, the structure equations are completely satisfied, despite the arbitrary functions $f(x)$ and $g(x)$. The proof of the case that $\epsilon =1$ is the same, but with slightly different PDE conditions in coordinates. Indeed, a similar analysis yields that 
\begin{equation*}
\begin{aligned}
\varphi(x,y,z) & =yf(x)+\csc(2x),\\
\psi(x,y,z) & =-2z\varphi(x,y)+2z^2/y+\xi(x,y)
\end{aligned}
\end{equation*}
for some $f(x)$ and for some $\xi(x,y)$ that satisfies $\omega^1\wedge d\omega^1+\omega^2\wedge d\omega^2 = 0$.
\end{proof}
\begin{prop}\label{normal-form-eta-pos}
Let the 1-forms $(\eta^1,\eta^2)$ and the open set $U$ with coordinates $(x,y,z)$ be as in Theorem \ref{normal-form} with $\epsilon=1$. Then with $\eta^3=y\,\omega^3$ the $C$-invariant for $(\eta^1,\eta^2)$ is
\[
C=\frac{\csc(2x)}{y}
\]
which satisfies $C_3\neq0$ and $dC\wedge dC_3\neq0$. Moreover, $U$ is the union of the following (possibly disconnected) open sets: 
\begin{equation}
\begin{aligned}
U_e&=\left\{(x,y,z)\in U: \left|\frac{\csc(2x)}{y}\right|<1\right\}\\
U_h&=\left\{(x,y,z)\in U: \left|\frac{\csc(2x)}{y}\right|>1\right\}.
\end{aligned}
\end{equation}

\end{prop}
\begin{proof}
We differentiate the relation between $\eta^1,\eta^2$ and $\omega^1$ and $\omega^2$. Let us first write 
\[
\bm{\omega}=R(x)\bm{\eta},
\]
where $\bm{\omega}=[\omega^1\quad \omega^2]^T$, $\bm{\eta}=[\eta^1\quad \eta^2]^T$, and 
\[
R(x)=\begin{bmatrix}\cos(x) & \sin(x)\\ -\sin(x) & \cos(x) \end{bmatrix}.
\]
It follows that
\[
d\bm{\eta}=R^{-1}d\bm{\omega}-R^{-1}dR\wedge \bm{\eta}.
\]
Moreover, if we let 
\[
S=\begin{bmatrix}C-A_3 & -1\\ -\epsilon  & -A_3-C \end{bmatrix}
\]
then
\[
d\bm{\omega}=S\omega^3\wedge \bm{\omega}
\]
and therefore 
\[
d\bm{\eta}=(R^{-1}SR-yR^{-1}R')\omega^3\wedge \bm{\eta},
\]
where $R'$ denotes the derivative of $R$ with respect to $x$. 
As such, when $\epsilon =1$ we get
\begin{equation*}
\begin{aligned}
d\eta^1&=yf(x)\eta^1\wedge\omega^3+y\eta^2\wedge\omega^3,\\
d\eta^2&=-y\eta^1\wedge\omega^3+(yf(x)+2\csc(2x))\eta^2\wedge\omega^3,
\end{aligned}
\end{equation*}
and so
\begin{equation*}
\begin{aligned}
\eta^1\wedge d\eta^1 & = y\eta^1\wedge \eta^2 \wedge \omega^3,\\
\eta^1\wedge d\eta^2 & =(yf(x)+2\csc(2x))\eta^1\wedge\eta^2\wedge \omega^3,\\
\eta^2\wedge d\eta^1 & = -yf(x)\eta^1\wedge \eta^2 \wedge \omega^3,\\
\eta^2\wedge d\eta^2 & =y\eta^1 \wedge \eta^2\wedge \omega^3.
\end{aligned}
\end{equation*}
Thus, if $\eta_a=a_1\eta^1+a_2\eta^2$ we find that 
\begin{equation*}
\begin{aligned}
\eta_a\wedge d\eta_a &=a_1^2\eta^1\wedge d\eta^1+a_1a_2\eta^1 \wedge d\eta^2+a_1a_2\eta^2 \wedge d\eta^1+a_2^2 \eta^2 \wedge d\eta^2,\\
\,&=(ya_1^2+ya_2^2-2a_1a_2(\csc(2x)))\eta^1\wedge \eta^2\wedge \omega^3. 
\end{aligned}
\end{equation*}
Notice that if we set $\eta^3=y\omega^3$ then the we have that 
\begin{equation*}
\begin{aligned}
\eta^1\wedge d\eta^1 & = \eta^1\wedge \eta^2 \wedge \eta^3,\\
\eta^1\wedge d\eta^2 & =\left(f(x)+\frac{2\csc(2x)}{y}\right)\eta^1\wedge\eta^2\wedge \eta^3,\\
\eta^2\wedge d\eta^1 & = -f(x)\eta^1\wedge \eta^2 \wedge \eta^3,\\
\eta^2\wedge d\eta^2 & =\eta^1 \wedge \eta^2\wedge \eta^3.
\end{aligned}
\end{equation*}
It then follows that 
\begin{equation*}
\begin{aligned}
\eta_a\wedge d\eta_a &=a_1^2\eta^1\wedge d\eta^1+a_1a_2\eta^1 \wedge d\eta^2+a_1a_2\eta^2 \wedge d\eta^1+a_2^2 \eta^2 \wedge d\eta^2,\\
\,&=\left(a_1^2+a_2^2-2a_1a_2\left(\frac{\csc(2x)}{y}\right)\right)\eta^1\wedge \eta^2\wedge \eta^3. 
\end{aligned}
\end{equation*}
So $(\eta^1,\eta^2)$ forms a bi-contact ellipse or hyperbola depending on when $|\frac{\csc(2x)}{y}|<1$ or $|\frac{\csc(2x)}{y}|>1$ respectively. 
\end{proof}
If one tries to extend Proposition \ref{normal-form-eta-pos} to the case of $\epsilon =-1$, it becomes clear that there is no way to normalize the induced volume forms from $\eta^1$ and $\eta^2$ to be consistent with a bi-contact ellipse. It is possible to adapt the frame further so that the $(\eta^1,\eta^2)$ bi-contact structure satisfies the case of $dC\wedge dC_3\neq 0$ and thus satisfies the structure equations \ref{struct-eqs 2}. The explicit formulas yield non-trivial and complicated $A_1$ and $A_2$ invariants which we do not list here. 

The following theorem yields another local normal form but for the case of $dC\wedge\omega^1\wedge \omega^2\neq0$ with $dC\wedge dC_3=0$. 
\begin{thm}
Let $(\omega^1,\omega^2,\omega^3,\omega^4)$ be a coframe on 4-manifold $\wt{M}$ with structure equations \eqref{symp-struct-eqs}. Then on any open set $U\subset \wt{M}$ such that $E>0$ there exists a local coordinate chart on a possibly smaller open set such that 
\begin{equation}
\begin{aligned}
\omega^1&=K\,dx+L\,dy,\\
\omega^2&=C\omega^1-\left(\partiald{K}{z}\,dx+\partiald{L}{z}\,dy\right),\\
\omega^3&=dz,\\
\omega^4&=\frac{1}{2w}dw-\left(K^2\partiald{\,}{z}\left(\frac{f}{K}\right)\right)dx-\left(L^2\partiald{\,}{z}\left(\frac{f}{L}\right)\right)dy,
\end{aligned}
\end{equation}
\end{thm}
where
\begin{equation*}
\begin{aligned}
f&=\frac{w}{W_0\det(h)}\left(\partiald{L}{x}-\partiald{K}{y}\right),\\
K&=\frac{1}{\sqrt{w}}\left(h_{11}(x,y)Q_1(z)+h_{12}(x,y)Q_2(z)\right),\\
L&=\frac{1}{\sqrt{w}}\left(h_{21}(x,y)Q_1(z)+h_{22}(x,y)Q_2(z)\right),
\end{aligned}
\end{equation*}
for some invertible $2\times 2$ matrix $h(x,y)=[h_{ij}]$ of smooth functions of two variables $x$ and $y$, the functions $Q_1$ and $Q_2$ are linearly independent solutions of
\[
\frac{d^2Q}{dz^2}=\left(C^2+\epsilon+\frac{dC}{dz}\right)Q,
\]
and $W_0\neq0$ is the (constant) Wronskian of $Q_1$ and $Q_2$ for some initial conditions. 
\begin{proof}
From the structure equations \eqref{symp-struct-eqs} on an open set $U\subset \wt{M}$ we observe that the Pfaffian system $\mcal{I}=\langle\omega^1,\omega^2\rangle$ is integrable by the Frobenius theorem. As such, there exists a local coordinate system $(x,y,z,w)$ on $U$ (or a possibly smaller open set) such that $\mcal{I}=\langle dx,dy\rangle$. Moreover, since $d\omega^3=0$, then we simply define the coordinate $z$ such that $\omega^3=dz$. Finally, we take coordinate $w$ such that $w=E$. This forms a a valid coordinate system since
\[
\omega^1\wedge\omega^2\wedge\omega^3\wedge\,dE=2E\,\omega^1\wedge\omega^2\wedge\omega^3\wedge\omega^4\neq0.
\]
We now express the coframe using this coordinate system. That is, there exist functions $K_i(x,y,z,w),L_i(x,y,z,w)$,$i=1,2$ such that 
\begin{equation}
\begin{aligned}
\omega^1&=K_1\,dx+L_1\,dy\\
\omega^2&=K_2\,dx+L_2\,dy.
\end{aligned}
\end{equation}
Moreover, we use the structure equation for $E$ to determine $\omega^4$. Let $E_i=2wf_i$, $i=1,2$ so that 
\[
\omega^4=\frac{1}{2w}dw-(f_1K_1+f_2K_2)dx-(f_1L_1+f_2L_2)dy.
\]
We now directly apply the structure equations \eqref{symp-struct-eqs} to these local coordinate expressions which yield the PDEs and algebraic conditions:

\begin{align}
2w\partiald{K_i}{w}+K_i&=0, & 2w\partiald{L_i}{w}+L_i&=0,\label{w-ODE}\\
\partiald{K_1}{z}-CK_1+K_2&=0, & \partiald{K_2}{z}+\epsilon K_1+CK_2&=0,\label{K-z-ODE}\\
\partiald{L_1}{z}-CL_1+L_2&=0, & \partiald{L_2}{z}+\epsilon L_1+CL_2&=0,\label{L-z-ODE}\\
-f_2(K_1L_2-K_2L_1)-\partiald{L_1}{x}+\partiald{K_1}{y}&=0, & f_1(K_1L_2-K_2L_1)-\partiald{L_2}{x}+\partiald{K_2}{y}&=0.\label{f-alg}
\end{align}
It is easy to find a general solution to these differential equations. Indeed, \eqref{w-ODE} yields $K_i=w^{-1/2}k_i(x,y,z),L_i=w^{-1/2}\ell_i(x,y,z)$ for $i=1,2$. Equations \eqref{K-z-ODE} and \eqref{L-z-ODE} are effectively linear ODE systems. In fact, if we differentiate twice the ODEs for $K_1$ and $L_1$ and use the first order equations for $K_2$ and $L_2$ we find that 
\begin{equation*}
\begin{aligned}
\partialdss{K_1}{z}&=\left(\frac{dC}{dz}+C^2+\epsilon\right)K_1,\\
\partialdss{L_1}{z}&=\left(\frac{dC}{dz}+C^2+\epsilon\right)L_1,
\end{aligned}
\end{equation*}
so that if $Q_1(z)$ and $Q_2(z)$ are linearly independent solutions of the ODE 
\[
\frac{d^2Q}{dz^2}=\left(C^2+\epsilon+\frac{dC}{dz}\right)Q
\]
then
\begin{equation*}
\begin{aligned}
K_1&=\frac{1}{\sqrt{w}}\left(h_{11}(x,y)Q_1(z)+h_{12}(x,y)Q_2(z)\right),\\
L_1&=\frac{1}{\sqrt{w}}\left(h_{21}(x,y)Q_1(z)+h_{22}(x,y)Q_2(z)\right),
\end{aligned}
\end{equation*}
where $h_{ij}(x,y)$ are arbitrary smooth functions so that the $2\times 2$ matrix $[h_{ij}]$ is invertible. Expressions for $K_2$ and $L_2$ follow from \eqref{K-z-ODE} and \eqref{L-z-ODE} yielding
\begin{equation*}
\begin{aligned}
K_2&=CK_1-\partiald{K_1}{z},\\
L_2&=CL_1-\partiald{L_1}{z}.
\end{aligned}
\end{equation*}
Next we observe that 
\[
K_1L_2-K_2L_1=-\frac{1}{w}\det(h)W(Q_1,Q_2),
\]
where $W(Q_1,Q_2)$ is the Wronskian of $Q_1$ and $Q_2$ which in fact is an arbitrary constant, say $W_0$, since $W'=0$. As such, 
\begin{equation*}
\begin{aligned}
f_1&=-\frac{w}{W_0\det(h)}\left(\partiald{L_2}{x}-\partiald{K_2}{y}\right),\\
f_2&=\frac{w}{W_0\det(h)}\left(\partiald{L_2}{x}-\partiald{K_2}{y}\right).
\end{aligned}
\end{equation*}
Upon substituting $L_2$ and $K_2$ we may write
\[
f_1=-Cf_2+\partiald{f_2}{z}.
\]
It then follows that 
\begin{equation*}
\begin{aligned}
f_1K_1+f_2K_2&=K_1^2\partiald{\,}{z}\left(\frac{f_2}{K_1}\right),\\
f_1L_1+f_2L_2&=L_1^2\partiald{\,}{z}\left(\frac{f_2}{L_1}\right).
\end{aligned}
\end{equation*}
Upon relabelling $K=K_1$, $L=L_1$, and $f_2=f$ the theorem follows. 
\end{proof}

\section{An Invariant Metric}\label{sec-inv-metric}
In all of the bi-contact structures studied in this paper so far, each of them admits a natural invariant Riemannian metric $g$ given by 
\[
g=(\omega^1)^2+(\omega^2)^2+(\omega^3)^2.
\]
We dedicate this section to some curvature calculations associated to this metric for each of the bi-contactomorphism cases. Notice that this metric need not be compatible with the bi-contact structure as in \cite{PerroneBiContactCirc,PerroneBiContactHyp}; however, in the case of a $(-1)$-bi-contact structure, this metric is exactly a minimizer of the Chern-Hamilton energy functional \cite{CritCompatMetricsChernHamilton}. 

\begin{remark}
We note that there is another natural invariant metric defined by 
\[
g_{\mcal{P}_C}=(\omega^1)^2+2C\omega^1\omega^2-\epsilon (\omega^2)^2+(\omega^3)^2.
\]
In this case the dual frame $(\be_1,\be_2,\be_3)$ fails to be orthonormal since $g_{\mcal{P}_C}(\be_1,\be_2)=C$. A related version of this metric will appear in forthcoming work. 
\end{remark}

In each case, we will be able to rewrite the invariant structure equations as the first Cartan structure equations for the Levi-Civita connection of the invariant metric. 

\[
\omega=\begin{bmatrix} \omega^1 \\ \omega^2 \\ \omega^3\end{bmatrix}
\]
so that each connection 1-form $\theta$ may be written as 
\[
\theta=\begin{bmatrix} 0 & \omega^1_2 & \omega^1_3\\ \omega^2_1 & 0 & \omega^2_3\\ \omega^3_1 & \omega^3_2 & 0\end{bmatrix}
\]
where $\omega^i_j =- \omega^j_i$ for all $1\leq i,j\leq 3$. so that the Cartan structure equations
\[
d\omega = -\theta \wedge \omega
\]
hold. In each subsection, we present the connection 1-form coming from the invariant structure equations for bi-contact equivalence. Then we'll present some curvature calculations of certain submanifolds via the curvature two form $\Theta=d\theta+\theta\wedge\theta$. We find in particular, that the invariant $C$ and $A_3$ has special significance in different cases. 

Assume that $\omega^3$ is integrable and let $\omega^i_j=\Gamma^i_{jk}\omega^k$ so that if $\Sigma$ is an integral manifold of $\omega^3$, then the second fundamental form for $\Sigma$ is
\begin{equation}\label{2nd ff}
\begin{aligned}
II_{\Sigma}&=\bar{\omega}^T\begin{bmatrix}\Gamma^3_{11} & \Gamma^3_{12}\\ \Gamma^3_{21} & \Gamma^3_{22} \end{bmatrix}\bar{\omega},\\
&=\bar{\omega}^TS_\Sigma\bar{\omega}
\end{aligned}
\end{equation}
where $S_\Sigma$ is the fundamental form matrix and $\bar{\omega}$ is the pullback of $\omega$ to leaves of the foliation of $\omega^3$. Thus, the mean curvature for each leaf of the $\omega^3$ foliation is $H=\Gamma^3_{11}+\Gamma^3_{22}$. Recall also, that 
\[
\Theta^1_2=d\omega^1_2+\omega^3_1\wedge\omega^3_2
\]
meaning the curvature 2-form of a leaf $\iota: \Sigma \hookrightarrow M$ of $\omega^3$ is  
\begin{equation}\label{leaf curvature}
\iota^*(d\omega^1_2)=\iota^*(\Theta^1_2)-\det(S_{\Sigma})\bar{\omega}^1\wedge\bar{\omega}^2.
\end{equation}
\subsection{Bi-contact Equivalence \texorpdfstring{$C_3=0$}{C3=0} Connection}
Let $\theta^{Bi_1}$ be the Levi-Civita connection 1-form for the $C_3=0$ case. Then $\theta^{Bi_1}$ has entries given by
\begin{equation*}
\begin{aligned}
\omega^1_2&=\frac{1}{2}(1-\epsilon -B_3)\omega^3,\\
\omega^3_1&=(A_3-C)\omega^1+\frac{1}{2}\left(1+\epsilon +B_3\right)\omega^2+B_2\omega^3,\\
\omega^3_2&=\frac{1}{2}\left(1+\epsilon -B_3\right)\omega^1+(A_3+C)\omega^2+B_1\omega^3.
\end{aligned}
\end{equation*}
where $B_1$ and $B_2$ are in terms of higher order derivatives of $C$, and $B_3$ is the obstruction to integrability of $\omega^3$. 
\begin{prop}
If $B_3=0$ then each leaf of the foliation defined by $\omega^3$ is flat as a submanifold of $M$ and has mean curvature
\[
A_3=C\cos(2\xi)-\frac{1}{2}(\epsilon +1)\sin(2\xi).
\]
\end{prop}
\begin{proof}
By equation \eqref{leaf curvature} and Theorem \ref{non-const-C-struct}, we deduce that 
\[
\iota^*(d\omega^1_2)=\frac{1}{2}(1-\epsilon )\iota^*d\omega^3=0
\]
and therefore each leaf of $\omega^3$ is flat. Now by \eqref{2nd ff} and since $\Gamma^3_{11}=A_3-C$ and $\Gamma^3_{22}=A_3+C$, we have that the mean curvature is $A_3$. 
\end{proof}
\begin{cor}\label{C3-zero cor}
For a bi-contact structure that satisfies Theorem \ref{non-const-C-struct} with $C_3=B_3=0$ everywhere, then the integral surfaces of $\omega^3$ are never minimal in $(M,g)$. 
\end{cor}
\begin{proof}
Apply the formula for $A_3=0$ and then deduce the implications of this from Proposition \ref{der-conditions}. 
\end{proof}
It is possible, however, that when $C$ is a constant (and hence Theorem \ref{non-const-C-struct} does not apply), we may have minimal surfaces as integral surfaces of $\omega^3$. 
\begin{ex}
Indeed, consider Example \ref{T^2xR ex} on $\mathbb{T}^2\times\mathbb{R}$. Then the structure equations are 
\begin{equation}
\begin{aligned}
d\omega^1&=-\sinh(2\psi)\omega^1\wedge \omega^3+\omega^2\wedge \omega^3\\
d\omega^2&=\omega^1\wedge\omega^3+\sinh(2\psi)\omega^2\wedge\omega^3\\
d\omega^3&=0.
\end{aligned}
\end{equation}
Thus, the connection mimics that of the case of $C_3\neq 0$. However, $A_3=0$ (which does not violate the structure equations for this case) unlike in Corollary \ref{C3-zero cor}. As such, each level surface of $\omega^3$ is minimal and flat and so we have that each torus $\mathbb{T}^2\times{z_0}$ is minimal and flat for all $z_0\in\mathbb{R}$.
\end{ex}
Moreover, it is straightforward to calculate the curvatures, which turn out to have all components zero or in terms of $C=\sinh(2\psi)$. The curvature 2-form is 
\[
\Theta=\begin{bmatrix} 0 & \cosh^2(2\psi)\,\omega^1\wedge \omega^2 & -\cosh^2(2\psi)\,\omega^1\wedge \omega^3\\
\cosh^2(2\psi)\,\omega^2\wedge \omega^1 & 0 & -\cosh^2(2\psi)\,\omega^2\wedge \omega^3\\
-\cosh^2(2\psi)\,\omega^3\wedge \omega^1 & -\cosh^2(2\psi)\,\omega^3\wedge \omega^2 & 0 \end{bmatrix}
\]

\subsection{Bi-contact Equivalence \texorpdfstring{$C_3 \neq 0$}{C3<>0} Connection}
Let $\theta^{Bi_2}$ be the Levi-Civita connection 1-form for the $C_3\neq0$ case. Then $\theta^{Bi_2}$ has entries given by
\begin{equation*}
\begin{aligned}
\omega^1_2&=A_2\omega^1-A_1\omega^2+\frac{1}{2}(\epsilon -1)\omega^3,\\
\omega^3_1&=(A_3-C)\omega^1+\frac{1}{2}\left(1+\epsilon \right)\omega^2+\sin(\zeta)\omega^3,\\
\omega^3_2&=\frac{1}{2}\left(1+\epsilon \right)\omega^1+(A_3+C)\omega^2+\cos(\zeta)\omega^3.
\end{aligned}
\end{equation*}
\begin{prop}
Each leaf of the foliation defined by $\omega^3$ has mean curvature $H=A_3$ and Gauss curvature $K=A_3^2-C^2-\frac{1}{2}(1+\epsilon )$.
\end{prop}
\subsection{Minimal Contact Hypersurfaces in a Symplectic Manifold}
Let us consider the situation of a bicontact structure such that $dC\wedge\omega^1\wedge\omega^2=C_3\Omega\neq0$ and $dC_3\wedge \omega^3=0$ on $M$. From Proposition \ref{symp-prop} we know that we may work on a symplectic manifold $\wt{M}$ where we have two independent symplectic structures given by $\theta^i=d\omega^i$, $i=1,2$. On $\wt{M}$ we have a natural invariant metric given by 
\[
g=(\omega^1)^2+(\omega^2)^2+(\omega^3)^2+(\omega^4)^2.
\]
As in the previous section, we may inquire as to the nature of some natural hyper surfaces. First we write out the connection 1-form $\theta$ so that $d\omega=-\theta\wedge \omega$ where $\omega=[\omega^1\quad \omega^2\quad\omega^3\quad\omega^4]^T$. If we denote
\[
\theta=\begin{bmatrix} 0 & \omega^1_2 & \omega^1_3 & \omega^1_4\\ \omega^2_1 & 0 & \omega^2_3 & \omega^2_4\\ \omega^3_1 & \omega^3_2 & 0& \omega^3_4\\ \omega^4_1 & \omega^4_2 & \omega^4_3 & 0\end{bmatrix}
\]
where $\omega^i_j =- \omega^j_i$ for all $1\leq i,j\leq 4$, then the structure equations $d\omega=-\theta\wedge\omega$ are compatible with our structure equations \eqref{symp-struct-eqs} if 
\begin{equation}
\begin{aligned}
\omega^1_2&=\frac{1}{2}(1-\epsilon )\omega^3-\frac{1}{2}E\omega^4,\\
\omega^3_1&=-C\omega^1+\frac{1}{2}\left(1+\epsilon \right)\omega^2,\\
\omega^3_2&=\frac{1}{2}\left(1+\epsilon \right)\omega^1+C\omega^2,\\
\omega^4_1&=\frac{1}{2}E\omega^2+\omega^1,\\
\omega^4_2&=-\frac{1}{2}E\omega^1+\omega^2,\\
\omega^4_3&=0.
\end{aligned}
\end{equation}
Integral surfaces of $\omega^3$ are contact submanifolds of $\wt{M}$ with contact form $\omega^4$. Moreover, the shape operator $S_\Sigma$ for such submanifolds is given by 
\[
S_\Sigma=\begin{bmatrix}-C & \frac{1}{2}\left(1+\epsilon \right) & 0\\\frac{1}{2}\left(1+\epsilon \right) & C & 0\\ 0 & 0 & 0 \end{bmatrix}.
\]
As such, we have that each $\Sigma$ an integral manifold of $\omega^3$ is minimal as a hypersurface in $(\wt{M},g)$. 

Given the relative simplicity of the connection, we may calculate the curvature two form. Let $\Theta$ be the $\mathfrak{so}(4)$-valued 2-form with nonzero entries given by
\begin{equation}
\begin{aligned}
\Theta^1_2&=(C^2-\frac{3}{4}E^2-\frac{1}{2}(1-\epsilon ))\omega^1\wedge\omega^2-\frac{1}{2}E_1\omega^1\wedge\omega^4-\frac{1}{2}E_2\omega^2\wedge\omega^4,\\
\Theta^1_3&=-\left(C^2+C_3+\frac{1}{2}(1+\epsilon )\right)\omega^1\wedge\omega^3+\left(C-\frac{1}{4}E(1+\epsilon )\right)\omega^1\wedge\omega^4\\
&+C(1-\epsilon )\omega^2\wedge\omega^3-\frac{1}{2}\left(1+\epsilon +CE\right)\omega^2\wedge\omega^4,\\
\Theta^1_4&=-\frac{1}{2}E_1\omega^1\wedge\omega^2+\left(C-\frac{1}{4}E(1+\epsilon )\right)\omega^1\wedge\omega^3+\left(\frac{1}{4}E^2-1\right)\omega^1\wedge\omega^4\\
&-\frac{1}{2}\left(1+\epsilon +CE\right)\omega^2\wedge\omega^3,\\
\Theta^2_3&=C(1+\epsilon )\omega^1\wedge\omega^3-\frac{1}{2}(1+\epsilon +CE)\omega^1\wedge\omega^4\\
&-\left(\frac{1}{2}(1+\epsilon )+C^2-C_3\right)\omega^2\wedge\omega^3+\left(-C+\frac{1}{4}E(1+\epsilon )\right)\omega^2\wedge\omega^4,\\
\Theta^2_4&=-\frac{1}{2}E_2\omega^1\wedge\omega^2-\frac{1}{2}(1+\epsilon +CE)\omega^1\wedge\omega^3+\left(-C+\frac{1}{4}E(1+\epsilon )\right)\omega^2\wedge\omega^3\\
&-\left(1-\frac{1}{4}E^2\right)\omega^2\wedge\omega^4,\\
\end{aligned}
\end{equation}
and $\Theta^3_4=0$. 

From the curvature 2-forms, we find that the scalar curvature is $S=-(E^2/2+2C^2+7+\epsilon )<0$. Moreover, we may compute $\text{Pf}(\Theta)$ and in so doing we find that 
\[
\text{Pf}(\Theta)=2((1+\epsilon )+2C^2)\Omega
\]
where $\Omega=\omega^1\wedge\omega^2\wedge\omega^3\wedge\omega^4$. Of course, $\wt{M}$ is an exact symplectic manifold, and hence there are no closed examples to which one may apply the Chern-Gauss-Bonnet theorem. Never-the-less, it may be possible to understand such manifolds with ends that admit finite volume and known extensions of the Chern-Gauss-Bonnet theorem. Indeed, since $dC\wedge \omega^1\wedge \omega^2 =C_3\,\Omega\neq0$ and $dC_3\wedge \omega^3=0$ then the invariant $C$ ought to generically be a Morse-Bott function, and its critical surfaces may well define $M$ as a manifold with ends. Understanding the corresponding $\wt{M}$ of such an $M$ with ends would be of potential interest in symplectic topology of 4-manifolds. 
%%%%%%%%%%%%%%%%%%%%%%%%%%%%%%%%%%%%%%%%%%%%%%%%%%%%%%%%%%%%%%%%%%%%%%%%%%%%%%%%%%%%
\section{Conclusion and Further Work}\label{sec-conclusion}

We have demonstrated that bi-contact structures under bi-contactomorphisms admit non-trivial local invariants. The invariant of primary concern is the function $C$ and its derivatives. We found local normal forms in special cases and demonstrated a Cartan style function count for initial data needed to determine a bi-contactomorphism in case $C$ is constant. Finally, we presented some invariant metric calculations and explored some local geometry of foliations determined by $\omega^3$.

There are many ways to generalize this study of the local geometry of bi-contact structures. One may increase the dimension and the number of independent contact structures and study similar types of diffeomorphisms as bi-contactomorphisms. The first author and the author of \cite{Jackman-Bi-Contact} are now working on projects that both generalize bi-contact geometry and find connections between bi-contact geometry and other research areas of intrigue. 

Finally, although this work is primarily local in nature, there are many interesting questions around finding compact examples of each of the bi-contactomorphism classes induced by an invariant $C$. Moreover, as hinted with the Pfaffian computation, it may be possible to discern global topological properties via the Chern-Gauss-Bonnet theorem for manifolds with ends determined by the invariant function $C$.

\bibliographystyle{amsplain}	% or "siam", or "alpha", etc.
%\nocite{*}		% list all refs in database, cited or not
\bibliography{ContactRef}	

\end{document}